\documentclass[11pt]{article}
\pdfoutput=1

\usepackage{listings}
\usepackage{fullpage}
\usepackage[left=1in,top=1in,right=1in,bottom=1in,headheight=3ex,headsep=3ex]{geometry}
\usepackage{graphicx}
\usepackage{enumitem}
\usepackage{amssymb}
\usepackage{amsthm}
\usepackage{mdframed}
\usepackage{relsize}

\newtheorem{Proposition}{Proposition}

  \newtheorem{Lemma}[Proposition]{Lemma}
   \newtheorem{Theorem}[Proposition]{Theorem}
\newtheorem{Note}[Proposition]{Note}
\newtheorem{Definition}[Proposition]{Definition}
\newtheorem*{Corollary*}{Corollary}
\title{Uniformization and Constructive Analytic Continuation of Taylor Series}
\author{Ovidiu Costin and Gerald V. Dunne \\~\\
{\small{\it Department of Mathematics, The Ohio State University, Columbus, OH 43210-1174, USA}}\\
\small{\it Department of Physics, University of Connecticut, Storrs, CT 06269-3046, USA}}

\usepackage[sc]{mathpazo}
\usepackage[T1]{fontenc}
\usepackage{url}
\usepackage{hyperref}\hypersetup{colorlinks=true,
 linktoc=all,}

\def\hhref#1{\href{http://arxiv.org/abs/#1}{arXiv:#1}} 
\usepackage{hyperref}

\usepackage{amsmath}
\newcounter{minisection}
\newcounter{minisection1}

\def\z{\noindent}
\usepackage{array}
\usepackage[dvipsnames]{xcolor}
\usepackage{soul}
\sethlcolor{SkyBlue}
\usepackage{hyperref}
\usepackage{cleveref}

\definecolor{clemsonorange}{HTML}{EA6A20}
\hypersetup{colorlinks,breaklinks,linkcolor=clemsonorange,urlcolor=clemsonorange,anchorcolor=clemsonorange,citecolor=black}

\def\RR{\mathbb{R}}

\def\CC{\mathbb{C}}
\def\NN{\mathbb{N}}

\def\ZZ{\mathbb{Z}}
\def\TT{\mathbb{T}}
\def\DD{\mathbb{D}}

\def\<{\langle}
\def\({\left (}
 \def\){\right)}
\def\epsilon{\varepsilon}
\def\phi{\varphi}

{\large  
\def\le{\leqslant}
\def\ge{\geqslant}

\def\>{\rangle}

\def\0{\emptyset}

}
\date{}

\begin{document}
\maketitle
\begin{abstract}

We analyze the problem of global reconstruction  of functions as accurately as possible,  based on {\it partial} information in the form of a truncated power series at some point, and additional analyticity properties. This situation occurs frequently in applications. 
  The question of the optimal procedure was open, and we formulate it as a well-posed mathematical problem.   Its solution leads to a  practical method which provides  dramatic accuracy improvements over existing techniques. Our procedure  is based on uniformization of Riemann surfaces. As an application, we show that our procedure can be implemented for solutions of a wide class of nonlinear ODEs.
We find a new uniformization method, which we use to construct the uniformizing maps needed for special functions, including solution of the Painlev\'e equations $P_{\rm I}$--$P_{\rm V}$.

We also introduce a new  rigorous and constructive method of regularization, {\it elimination of singularities} whose position and type are known. If these are unknown, the same procedure enables a highly sensitive resonance method to determine the  position and type of a singularity.

In applications where less explicit information is available about the Riemann surface, our approach and techniques  lead to new approximate, but  still much more precise reconstruction methods than existing ones, especially in the vicinity of singularities, which are the points of greatest interest.

\vspace{0.2cm}

\z{\bf MSC} codes:  30F20, 30B30, 30B10, 40A25, 41A58
\end{abstract}

\section{Introduction}
\label{sec:intro}

In problems of high complexity in mathematics and physics it is often the case that a solution can only be generated as a {\it finite} number $n$  of terms of a perturbation series at certain special points, convergent or, more often, divergent but generalized Borel summable.
Under these circumstances, we ask what is the {\it optimal} strategy to approximate the underlying function, and if optimality cannot be achieved in practice, what are the most efficient near-optimal methods?  The mathematical question of optimality was open, and here we prove a result that is of practical interest as well as being highly accurate. On a practical level, this reconstruction question has been encountered in many problems in the literature, and dealt with in various problem-specific ways \cite{fisher,baker,guttmann,ZinnJustin:2002ru,caliceti,caprini}, not necessarily with optimal accuracy and without a rigorous mathematical foundation.

Clearly, without any further information about a function $F$, the question of optimally reconstructing it from its truncated series is ill posed. However, we show here that the question is well-posed within  generic classes of functions analytic on a given Riemann surface $\Omega$, and with a given rate of growth on $\Omega$.  

The type of information we need to achieve optimality ($\Omega$ and the rate of growth) is known a priori in the {\em Borel plane}  of solutions of generic linear or nonlinear meromorphic ODEs or difference equations, of classes of PDEs including the one-particle time-independent or time-periodic Schr\"odinger equation and many more.  This is a result of  \'Ecalle's pioneering theory of resurgent functions \cite{Ecalle}. Furthermore, resurgence  theory provides a wealth of a priori information about the Borel plane singularity structure \cite{Ecalle,Berry,costin-book,Sauzin,Garoufalidis:2010ya,kontsevich,gokce}.

The information required by our methods is also known or conjectured in many physics models such as in quantum mechanics, random matrix theory, quantum field theory, string theory  \cite{voros,delabaere,Delabaere,Berry,ZinnJustin:2004ib,marino-matrix,Gaiotto:2009hg,Garoufalidis:2010ya,Aniceto:2011nu,Dunne:2012ae,Dunne:2016nmc,Dorigoni:2015dha,Aniceto:2015rua,Gukov:2016njj}. The new analysis tools introduced in this paper can be used  to corroborate or refine such conjectural information  with high precision.

Finally, if the information is unknown, then it can be found empirically (nonrigorously) to very high precision by the methods we introduce.

Relative to this information, in this paper we find:
 \begin{itemize}
 	\item[--]
 	The explicit optimal reconstruction formula of a function, given  $n$ coefficients of its Maclaurin series (Theorem \ref{T1},  in \S \ref{sec:optimal}). This optimal reconstruction  is based on uniformization maps of Riemann surfaces.
The accuracy of the optimal procedure is often dramatically better than currently used methods. Some examples are discussed in \S\ref{Saccsing}, \S\ref{S:ExampleP1} and \S\ref{S:Applic}.

 	\item[--] Uniformization formulas for Riemann surfaces commonly encountered in applications (\S\ref{sec:borel-odes}).
 	
 	\item[--] A singularity elimination method transforming singularities whose position and type are known into regular points. After a singularity is eliminated, local Taylor series map back into the singular expansions of interest (\S\ref{sec:elim}).  If the singularity is unknown, the same procedure gives rise to a highly sensitive resonance method to determine the position and type of a singularity; this will be discussed in a separate paper.
 	
 	\item[--] Refinements of classical techniques such as Pad\'e approximants (\S\ref{S:empirical}).
 \end{itemize}
In mathematics, some of the major questions of interest where our methods bring substantial improvement are  in: 

\begin{itemize}

\item[--] Reconstructing functions globally from their asymptotic expansions at a point.

\item[--] Determining the singularities on higher Riemann sheets.
	
\item[--] Finding local expansions at singularities, including on higher Riemann sheets to determine precise connection formulae.
	
\item[--] Relatedly,  accurately calculating Stokes constants.
	
\item[--] Extrapolating additional Maclaurin coefficients, beyond the given number $n$.
	
\end{itemize}

In physics applications, important types of questions to be answered are:
\begin{enumerate}
\item
Where are the singularities of $F$, if the Riemann surface is unknown?

\item
What is the local behavior
at these singularities?
\item
How far can one explore the full Riemann surface of $F$ as a function of $n$?
\item
Can one quantify the expected precision locally, especially near the singularities, as a function of how much input data (and of what precision) is given?  What is the possible accuracy of reconstruction, especially near singularities, again as a function of $n$?
\end{enumerate}
The answers to these questions have both theoretical and practical consequences.
Question 1 corresponds to identifying critical points or saddle points (e.g. for asymptotics and phase transitions). Question 2 refers to determining whether these singularities are algebraic or logarithmic branch  points (or in special simple cases, poles), or essential singularities. A common application in statistical physics and quantum field theory is the accurate determination of critical exponents \cite{ZinnJustin:2002ru}. Questions 2 and 3 involve for example the numerical determination of Stokes constants, or wall-crossing formulas \cite{Gaiotto:2009hg}, or generally the fluctuations about a given critical point. Question 4 is of particular practical value, since in nontrivial applications it is often difficult to generate many terms of the original series.

 {\em Overview of the paper.} In \S \ref{sec:optimal} we prove the optimality theorem, Theorem \ref{T1}. In \S \ref{sec:borel-odes} we construct the explicit uniformization of the Riemann surfaces of the Borel plane of ODEs, both linear and non-linear. 
We also find a new and constructive uniformization procedure, which is geometric in nature. Theorem \ref{Th:compo} (in \S\ref{sec:composition}) expresses this uniformization in terms of an infinite composition of elementary maps, whose truncations provide accurate {\it approximate} uniformization  by elementary conformal maps. 
In \S \ref{sec:P1-unif} we give the uniformization of the 
Borel plane of the Riemann surface of {\it tronqu\'ee} solutions of the Painlev\'e equations $P_{\rm I}$--$P_{\rm V}$.  This procedure extends to more general nonlinear ODEs with sufficient symmetry properties. We  illustrate the dramatic gains in precision with some examples: see \S \ref{sec:t1discussion} and \S \ref{Scoefextrap}.

The use of uniformizing maps is a practical tool (and perhaps the only one that does not require full knowledge about the function) to eventually access all Riemann sheets. This kind of information is difficult to obtain numerically in other ways (cf. \S\ref{sec:odes} and \S\ref{sec:P1-unif}). 

Constructing {\it exact} uniformization maps for more complicated Riemann surfaces can be a challenging task, but we show that approximate maps also lead to dramatic improvements in the precision of the reconstruction of the function $F$.
In such cases 
the procedure in Theorem \ref{T1} can be adapted to extract information  using simpler maps with near-optimal precision in restricted regions of $\Omega$, for example near the singularities or boundaries.
Discrete  singularities and natural boundaries are usually the most important regions to study.

In \S \ref{sec:elim} we prove a {\it Singularity Elimination} theorem (Theorem \ref{thelim}), a rigorous construction of invertible linear operators that regularize a chosen singularity, transforming it into a point of analyticity. 
New approximate and exploratory numerical methods that can be used to probe the singularity structure of the function to be reconstructed are the object of \S\ref{S:empirical}.
We combine our methods with Stahl's fundamental analysis and results on Pad\'e approximation \cite{Stahl}, to propose methods for {\em approximate uniformization}, resulting in significantly improved analytic continuation methods. 
In applications, the methods we present can be used as precise ``discovery tools'' to explore empirically Riemann surfaces even when they are not known {\it a priori}, or are conjectured and need numerical validation.  \S\ref{S:Applic} contains some applications of our new methods, and the Appendix \S \ref{sec:app} lists some special conformal maps which are used in our analysis.

{ \em Illustrative examples.} Throughout the paper we provide examples showing the power of these optimal (and almost-optimal) methods.
In \S\ref{Scoefextrap} we consider functions analytic on the Riemann surface  in Borel plane of special functions solving linear ODEs  for which a Maclaurin polynomial  is given, truncated at degree 8, $P_9$ ($9$ possibly nonzero coefficients). From this truncated polynomial $P_9$,  one can extrapolate to the polynomial $P_{481}$ with relative errors in the coefficients of at most $\sim 0.1\%$. And from the same $P_9$  the value of the function at points on a circle of radius $3R$ (where $R$ is the radius of convergence of the Maclaurin series) is calculated with errors $<10^{-6}$, whereas Pad\'e approximants give 100\% errors at about $1.2R$. Pad\'e approximants, usually quite suboptimal, prove an important phenomenon: the mere existence of a larger domain of analyticity, without a priori knowledge about it, allows for series extrapolation, cf. \S\ref{Scoefextrap}. The improvement is proportional to the  ``size'' of $\Omega$ (for Pad\'e $\Omega$ is of course always a domain in $\CC$). 

The efficacy of analytic continuation is seen in  \S\ref{sec:odes} where the polynomial $P_{200}$ of an elliptic integral gives access to singularities located on many tens of sheets of its Riemann surface, cf. Fig. \ref{fig:Ksurface}. We are not aware of any other method that can achieve that.

The global reconstruction efficacy is illustrated in  \S\ref{S:ExampleP1} on the tritronqu\'ee solution of Painlev\'e $P_{\rm I}$ where the pole sector centered on $\RR^-$ is reconstructed from 200 terms of the asymptotic expansion at the opposite end, $+\infty$, accurately recovering the 66 poles closest to the origin, the first one with more than 60 digits of accuracy. 

In  \S\ref{Saccsing} it is seen that the power of the optimal method is even more dramatic near singularities, i.e., near $\partial\Omega$. As seen in Fig. \ref{fig:Ksurface-loops}, optimally used, $P_{200}$ gives access to a neighborhood of size $\sim 10^{-52}$ of a singular point, to find the type of the singularity and the local behavior. For the same neighborhood without any optimization one would need $P_{10^{52}}$, which is clearly impossible to obtain and/or use.

Note \ref{unifp1}.\ref{allsheets} shows that in a precise sense, with the optimal method one can see the influence of infinitely deep/high sheets of $\Omega$. Using $P_{200}$ for the Borel transform of the tronqu\'ee solutions of $P_{\rm I}$, we detect exponential behavior on $\Omega$, which is known to exist, indeed, only on infinitely deep/high sheets of $\Omega$.

There is a wealth of information that can be decoded from a given Maclaurin polynomial. Further illustrative examples of the precision achieved by our methods are given in  \S\ref{sec:odes}, \S\ref{sec:t1discussion}, \S\ref{Scoefextrap} and \S\ref{S:Applic}.

\subsection{Settings and Notations of the Paper}
\label{sec:summary}

In the following, $\DD_r(a)$ denotes the open disk of radius $r$, centered at $a$. The unit disk centered at the origin appears often in the discussion, and is simply denoted $\DD=\DD_1(0)$, with boundary the unit circle: $\TT=\partial\DD$. The Riemann sphere is written as $\hat{\CC}=\CC\cup\{\infty\}$.

We consider functions defined on  $\Omega$,  a simply connected Riemann surface.\footnote{Non-simply connected Riemann surfaces are uniformized on $\CC/\Gamma$ or $\DD/\Gamma$, where $\Gamma$ is a discrete group of automorphisms, see e.g. \cite{Schlag}. At this stage it is unclear to us how to take advantage of the factorization and in such a case we take instead $\Omega$ to be the universal covering.\label{ftn4}} An important special case in applications is $\Omega$ being a simply connected domain strictly contained in $\CC$, in which case  the uniformization map is its usual Riemann conformal map to $\DD$.
\begin{Definition}\label{LD1}
  {\rm
      We denote by $\psi$ the conformal map of $\Omega$ onto $\DD$, uniquely specified by the normalization $\psi(0)=0,\psi'(0)>0$,  and we write the inverse map $\psi^{-1}=\phi$, the covering map of $\Omega$.  See Figure \ref{fig:maps}.

}
\end{Definition}

 \begin{figure}\centering 
	\includegraphics[scale=0.75]{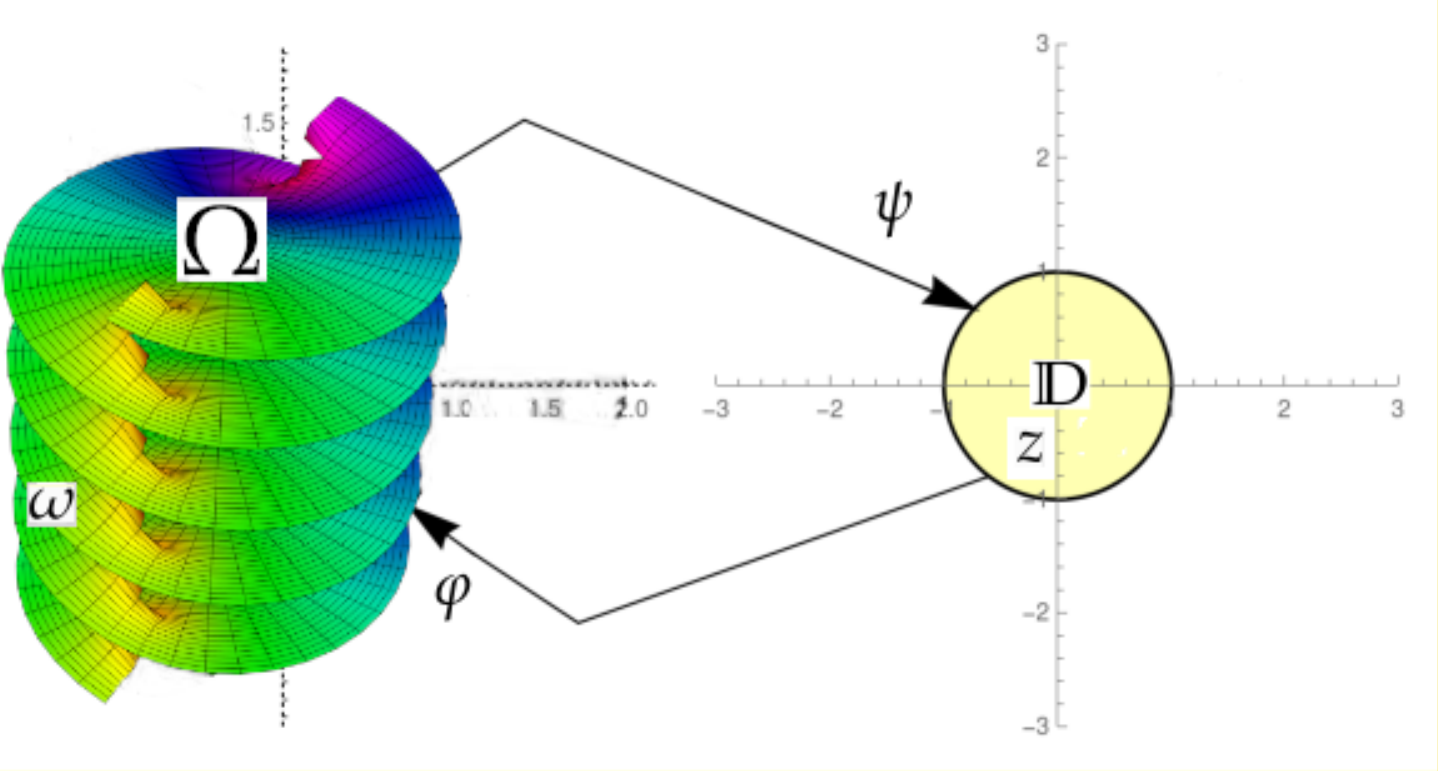}
	\caption{
		The map $z=\psi(\omega)$ from the simply connected Riemann surface $\Omega$ (here the universal cover of $\hat\CC\setminus\{\pm i,\infty\}$) to the unit disk $\DD$, and its inverse $\omega=\phi(z)$.} 
		\label{fig:maps}
	\end{figure}

\begin{Note}\label{NoteN1}{\rm 
	\begin{enumerate}
		\item\label{N2i} 	 For many functions of interest, the point of expansion, say zero, may become singular on other Riemann sheets.  
		In this special case,  $\Omega$ is described by equivalence classes of curves {\bf originating at zero, modulo homotopies in $\hat \CC\setminus S$},\footnote{Using $\hat{\CC}$ is a standard convention, since it makes a counting difference for the analyzed functions if infinity is singular or not. For instance, $\ln[(1-\omega)/(1+\omega)]$ is analytic at infinity and its Riemann surface is uniformized on the plane, after a M\"obius change of variable.}  where $S$ is a discrete set  that may contain zero. We shall call such Riemann surfaces coverings with fixed origin. The disk of analyticity at zero will be normalized to $\DD$.
		
	\item In view of \ref{N2i}, $ \Omega$ will be assumed to {\bf contain $\DD$  strictly}. More precisely, if $\Omega$ is uniformized to $\DD$ by $\psi$, then $\psi$ is analytic in $\DD$ and $\psi(0)=0$.
	\end{enumerate}

  }\end{Note}
  \begin{Definition}\label{DOmega()}
    { \rm We denote by $\Omega\left(\hat\CC\setminus P\right)$  Riemann surfaces described by homotopy classes over $\hat\CC\setminus P$, where $P$ is a discrete set of punctures. When the curves have fixed origin as in Note \ref{NoteN1}, say at $0$, the Riemann surfaces are denoted by $\Omega\left(\hat\CC\setminus\{(0),P\}\right)$.

  As an example, an elementary function that lives on $\Omega(\hat\CC\setminus\{(0),1,\infty\})$ is $\log(\omega^{-1}\log(1-\omega))$:  on the first sheet it has two branch points, $\{1, \infty\}$,  and three branch points,  $\{0, 1, \infty\}$, on all other sheets. This is also the Riemann surface of the complete elliptic integral of the first kind $F(\omega)=\mathbb K(\omega)$,  which is a test case analyzed in Sections \ref{sec:odes}, \ref{sec:t1discussion}, \ref{Scoefextrap}, \ref{sec:sing-elim-ex}.}

  \end{Definition}

By the uniformization theorem, $\Omega$ is biholomorphically equivalent to exactly one of the following: $\DD, \CC$ or $\hat{\CC}$ (see, e.g. \cite{Abikoff,Ablowitz,Schlag}). Here, we mostly focus on Riemann surfaces uniformized on $\DD$, as the latter two cases are too special,  occurring only in the simplest cases  \cite{Abikoff}, and also because their analysis would follow similar steps.

It will at times be convenient to consider  shifts of sets in $\CC$, in which case we write 
$S+\zeta=\{\omega+\zeta:\omega\in S \subset \CC\}$, and also to work with an inverted variable, changing the expansion point from $\omega=0$ to $\omega=\infty$,  in which case we write $1/S=\{1/\omega:\omega\in S \subset \CC\}\subset \hat{\CC}$.

In our discussion of resurgent functions, the singularities assume a simple form (cf. \cite{Duke})
\begin{equation}
 (\omega-\omega_0)^{\alpha} A(\omega)+B(\omega);\ \alpha\in \CC\setminus \ZZ;\ \ \text{or, for } \alpha \in \ZZ,\ \ \  \frac{d^k}{d\omega^k}[(\omega-\omega_0)^{\alpha}\ln (\omega-\omega_0) A(\omega)]+B(\omega)
   \label{eq:res-sing}
\end{equation}
 for some $\alpha \in\CC$, $k\in\NN$, and where  $A,B$ are analytic at $\omega_0$ and $A(\omega_0)\ne 0$. In this paper we refer to singularities of the type in (\ref{eq:res-sing}) as {\it elementary singularities}. With this notation, important goals of our analysis are to learn as much as possible about the singularity locations $\omega_0$, the singularity exponent $\alpha$, and the associated local functions $A(\omega)$ and $B(\omega)$.

\section{Optimal Reconstruction}
\label{sec:optimal}
We start with a discussion of the underlying question and the various ideas involved in optimal reconstruction. We place ourselves in a frequently encountered setting in which a Maclaurin polynomial $P_n$  is given, together with the underlying type of Riemann surface $\Omega$, combined with some a priori weighted bounds:
  \begin{equation}
    \label{eq:space1}
    \mathcal{F}_{P_n,W}=\{F \text{ analytic in } \Omega|F_n=P_n+o(z^n);\ \|F\|_W:=\|F W\|_{\infty}<\infty\} 
  \end{equation}
  Here we consider weights that may allow {\em growth} of the functions involved; thus, $W:\Omega\to (0,1]$. Furthermore, it is natural to consider weights that depend on the {\em conformal distance to the boundary},  $\rho(\omega):=1-|\psi(\omega)|$. Then, $W(\omega)=W(\rho(\omega))$.
  \begin{Definition}
  \label{def:ck}
    {\rm  We denote $P_n(z)= \sum_{k=0}^{n-1}p_k z^k$, and $\phi^{(k)}(0)/k!=c_k$.} 
  \end{Definition}
  \begin{Note}\label{Nlocal}
 {\rm   Any approximant based on $P_n$ at a point $\omega_0\in \Omega$ for a class of functions $\mathcal{F}_{P_n,W}$, is a function
$R_n=R_n\left(\omega_0\right)$
   and the {\bf best approximant} $\hat{R}_n(\omega_0)$,  in the class $\mathcal{F}_{P_n,W}$, is defined  by minimizing the scaled quantity
  \begin{equation}
    \label{eq:defoptimal}
 \sup_{F\in\mathcal{F}_{P_n,W}} \frac{|F(\omega_0)-R_n(\omega_0)|}{\|F\|_{W}} 
  \end{equation}
  over the set of all possible $R_n$s.  
 } \end{Note}

The question is thus to reconstruct $F$ at any $\omega_0\in \Omega$ with optimal accuracy  in the class $\mathcal{F}_{P_n,W}$. Relatedly, the same question is important in contexts when only partial information about $\Omega$ and bounds is available. The optimality questions are relative to the {\bf whole} class $\mathcal{F}_{P_n,W}$.

As a mathematical question, this is one of {\it inverse approximation theory}, in the sense that here the approximation is fixed, in the form of a number $n$ of terms of a series, and the underlying function $F$ is to be reconstructed as accurately as possible. 

Theorem \ref{T1} shows that the best approximant is given, using the notations of \S\ref{sec:summary} (and recall Figure \ref{fig:maps}), by the composition
  \begin{equation}
    \label{eq:optimalappr}
    \hat{R}_n=(P_n\circ\phi)_n\circ\psi
  \end{equation}

 It {\em may seem paradoxical} that, in order to extract the most information from $P_n$, Theorem \ref{T1} shows that some information needs to be {\bf discarded} (by the truncation $P_n\circ\phi\mapsto (P_n\circ\phi)_n$); at a heuristic level, this is explained in Note \ref{NExpGrowth}, below.

Assume for the moment that our weight is $W=1$, and let us analyze the ball of a given radius $A$ compatible with $P_n$, say  $\mathcal{F}_{A}=\{F\in\mathcal{F}_{P_n,\infty}|\|F\|_{\infty}\le A\}$. Then, for $F\in \mathcal{F}_A$, the maximal error at a point $\omega_0\in \DD\subset \Omega$ obtained using $P_n$ as an approximant is (using Cauchy estimates) bounded by
  \begin{equation}
    \label{eq:origerr}
|F(\omega_0)-P_n(\omega_0)|\le    |A|\frac{|\omega_0|^n}{1-|\omega_0|}
\end{equation}
while  Theorem \ref{T1} shows that for all $\omega_0\in \Omega$
\begin{equation}
    \label{eq:origerrcomposed}
|F(\omega_0)-(P_n\circ\phi)_n\circ\psi(\omega_0)|\le    |A|\frac{|\psi(\omega_0)|^n}{1-|\psi(\omega_0)|}
\end{equation}

This provides provides exponential improvement of accuracy:
\begin{Note}\label{LSchwarzLemma}
  {\rm   By Note \ref{NoteN1} \ref{N2i}., $\DD\subsetneq \Omega$ and hence $\psi(\DD)\subsetneq\DD$. By Schwarz's lemma, $|\psi(\omega)|< |\omega|$ for all $\omega\in\DD$; the factor  $|\psi(\omega)/\omega|$ (extended by $\psi'(0)$ at zero), plays the role of an {\em accuracy} {\bf acceleration modulus}. It is a crucial quantity throughout the analysis.

 Furthermore, while $P_n$ diverges outside $\DD$, we have instead $|F(\omega_0)-(P_n\circ\phi)_n\circ\psi(\omega_0)|\to 0$ as $n\to\infty$ {\bf throughout $\Omega$} (since $\psi(\Omega)=\DD$)}.
 	
\end{Note}
 
\begin{Note} [Monotonicity of $\phi'(0)$] \label{monot}{\rm 
When uniformization of the whole Riemann surface $\Omega$ is impractical, one should map to $\DD$ as much of $\Omega$ as possible. Indeed, if $\Omega_2\subset \Omega_1$, then $\displaystyle{ \DD\mathop{\to}^{\phi_2}\Omega_2\mathop{\to}^{\phi_1^{-1}}S\subset \DD}$. By Schwarz's lemma $\phi_2'(0)/\phi_1'(0)\le 1$ and hence $\phi'(0)$ is increasing in the size of $\Omega$.
			This also follows from the Optimality Theorem  \ref{T1} below.
}\end{Note}

\subsection{The Optimal  Reconstruction Theorem on a Riemann Surface} 
\label{sec:thm1}

Let $\Omega$ be a Riemann surface as in \S\ref{sec:summary} and $\omega_0\in \Omega$. The following result constructs and characterizes $\hat{R}_n(\omega_0)$, the best approximant, in the sense stated in  Theorem \ref{T1}, at $\omega_0$ within the class of functions $\mathcal{F}$ analytic on the same Riemann surface  $\Omega$ and a common Maclaurin polynomial $P_n$, which is our input data.
Recall the maps $z=\psi(\omega)$ and its inverse $\omega=\phi(z)$ in Figure \ref{fig:maps}. 

Theorem \ref{T1} shows that the best approximants are given, using the notations of \S\ref{sec:summary}, by the composition $\hat{R}_n=(P_n\circ\phi_n)_n\circ\psi$.    Part 2 of Theorem \ref{T1}  shows that with optimality even the constant in the optimal bound becomes sharp, if the functions are already ``well approximated by this procedure''. Part 3 allows for weighted bounds, that is for growth towards $\partial\Omega$.

\begin{Theorem}[{\bf The optimality theorem}]\label{T1} 
Let $\Omega$ be a Riemann surface as in \S\ref{sec:summary}, $\omega_0\in \Omega$, and  $P_n$ an $(n-1)$-order truncation of a Maclaurin series and let $\hat{R}_n= (P_n \circ \phi)_n\circ\psi$. We denote by ${\mathcal{F}}_{P}$ the set of bounded functions on $\Omega$ to which $P_n$ converges:
$$ \mathcal{F}_{P}=\left\{ F\in{\mathcal{F}}:  {|| F\|_{\infty}}<\infty, \text{ and }F(\omega)-P_n(\omega)=O(\omega^{n}) \text{ as }\omega\to 0\right\}$$
Let $\omega_0\in \Omega$. Then,
\begin{enumerate}
     \item For $F\in \mathcal{F}_P$ we have
    \begin{equation}
      \label{eq:CauchyEst}
     \frac{ |F(\omega_0)-\hat{R}_n(\omega_0)|}{  \|F\|_{\infty}}\le \frac{|\psi(\omega_0)|^n}{1-|\psi(\omega_0)|}
      \end{equation}

 For every $R\in\CC$ and $\delta>0$ there exists $F_\delta\in \mathcal{F}_{P}$ so that 
      \begin{equation}
    \label{eq:optimal-extrapolation}
   \frac{|F_\delta(\omega_0)-R|}{\|F\|_{\infty}}\ge |\psi(\omega_0)|^n(1-\delta)
  \end{equation}
In this sense the reconstruction $\hat{R}_n$ is optimal.
\item Furthermore, for $\epsilon>0$ let
$$\mathcal{F}_\epsilon= \{F\in \mathcal{F}_{P}: F(\omega_0)\ne 0 \quad \text{and}\quad |F(\omega_0)|^{-1}|F(\omega_0)-\hat{R}_n(\omega_0)|\le \epsilon \}$$ 
We have
\begin{equation}\label{isepsi}
  \sup_{F\in \mathcal{F}_\epsilon} \frac{|F(\omega_0)-\hat{R}_n(\omega_0)|}{|F(\omega_0)|}=\epsilon
  \end{equation}
  
  Assume $n$ is large enough so that $(1-|\psi(\omega)|)^n<\epsilon$. Then for every $R\in \CC$ and every $\delta>0$ there exists an 
  $F_\delta\in \mathcal{F}_\epsilon$ so that 
  \begin{equation}
    \label{eq:optimal-extrapolation2}
   \frac{|F_\delta(\omega_0)-R|}{|F_\delta(\omega_0)|}\ge \frac{1-\delta}{1+2\epsilon} \, \frac{|F_\delta(\omega_0)-\hat{R}_n(\omega_0)|}{|F_\delta(\omega_0)|}
  \end{equation}
In this sense the reconstruction $\hat{R}_n$ is optimal, also including constants.
\item Let $W=W_1\circ\psi$, where $W_1:[0,1)\to \RR^+$ (a weight depending on the natural metric distance ``to the boundary''). 
With  $\|\cdot\|_W$ defined as in (\ref{eq:space1}), 
let $\mathcal{F}_{W}$ be the family of functions $F$ analytic in $\Omega$ and such that $\|F\|_{W}<\infty$. 
Then,
  \begin{equation}
    \label{eq:generalestimate}
  \frac{ |F(\omega_0)-\hat{R}_n(\omega_0)|}{  \|F\|_{W}}\le |\psi(\omega_0)|^n \inf_{r\in(|\psi(\omega_0)|,1)} \frac{1}{W(r)r^{n-1}(r-|\psi(\omega_0)|))}
\end{equation}
and  for every $R\in\CC$ and $\delta>0$ there exists $F_\delta\in\mathcal{F}_{P_n,W}$
such that 
      \begin{equation}
        \label{eq:optimal-extrapolation3}
   \frac{|F_\delta(\omega_0)-R|}{\|F\|_{W}}\ge |\psi(\omega_0)|^n\inf_{r\in(|\psi(\omega_0)|,1)} \frac{(1-\delta)}{W(r)r^{n}}
  \end{equation}
   \end{enumerate}
 \end{Theorem}
\begin{Note}{\rm 
    \begin{enumerate}
    \item Observe that the map $F\mapsto F\circ\phi$  is an isometric isomorphism taking the space of analytic functions in  $\Omega$ onto the space of analytic functions in $\mathbb{D},$\footnote{Analyticity follows from the fact that $\phi(\DD)=\Omega$ and $F$ is analytic in $\Omega$.}  $L^\infty(\Omega)$ onto $L^\infty(\mathcal{\DD})$, and $\mathcal{F}_W$ to $\mathcal{F}'_W$,  the space of functions in $\DD$ for which $\sup_{z\in\DD}|F\circ\phi(z)/W(|z|)|$ is finite.
      
    \item An immediate calculation shows that there is an explicit  bijection between the polynomials $F_n$ and $(F\circ\phi)_n$. Since polynomials are bounded in $\DD$, there is a simpler characterization of $\mathcal{F}_{P_n,W}$ when mapped to $\DD$:
      \begin{equation}
    \label{eq:space}
    \mathcal{F}_{P_n,W}=\{F|F\circ\phi=(P_n\circ\phi)_n+z^n g;\ g\text{\ analytic in }\DD; \ \|gW\|<\infty\} 
  \end{equation}
    
    \end{enumerate}
   }
\end{Note}

  \begin{proof}
    We take  $F_{\lambda}=(P_n\circ\phi)_n\circ\psi+\lambda \psi^n$. It is straightforward to check that the Maclaurin polynomial of $F_\lambda$ starts with $P_n$ and that $F_\lambda$ is bounded, implying $F_\lambda\in\mathcal{F}_{P_n,W}$. Taking $\omega\to \partial \Omega$ we have $|\psi(\omega)|\to 1$ and we see that $\lambda^{-1}\|F\|_{\infty}\to 1$. Hence,  for any approximant $R$ based on this data and any $\omega_0$ we have
  \begin{equation}
    \label{eq:no-furter-info}
    \|F_\lambda\|_{\infty}^{-1}|F_\lambda(\omega_0)-R|= \frac{|\lambda^{-1} P_n(\omega_0)+ \psi(\omega_0)^n-\lambda^{-1}R|}{\| \lambda^{-1}F\|_{\infty}} \mathop{\to}_{\lambda\to \infty} |\psi(\omega_0)|^n
  \end{equation}
  showing that, in the limit $\lambda\to \infty$, $|\psi(\omega_0)|^n$ is a lower bound of the approximation achievable by any $R$, proving \eqref{eq:optimal-extrapolation}.

  In fact, due to the biholomorphic bijection between $\Omega$ and $\DD$ we can map all needed inequalities back and forth between $\Omega$ and $\DD$. The domain $\DD$ is however simpler to handle, and we will map our questions to $\DD$  from this point on. We denote $\psi(\omega_0)=z_0$ and $F\circ \psi=f$. Mapped to $\DD$, our space is
$$\{f:f\circ\phi\in\mathcal{F}_{P_n,W}\}$$

      The inequality  \eqref{eq:CauchyEst} is simply obtained by taking the sup in the Cauchy formula
    \begin{equation}
      \label{eq:CauchyFormula}
      f(z_0)-\hat{R}(z_0)=\frac{z_0^n}{2\pi i}\oint\frac{s^{-n} f(s)}{s-z_0}ds
    \end{equation}
where the contour of integration is $\partial \mathbb{D}_{r}(0)$ with $r\in(|z_0|,1)$ and  letting $r\to 1$. 
 
To prove \eqref{isepsi}, it suffices to take $F(\omega)=\hat{R}_n(\omega)+c \psi(\omega)^n$, with $c=\epsilon \hat{R}(\omega_0)\omega_0^{-n}/(1-\epsilon)$. We again map the question to $\DD$.
 
  For \eqref{eq:optimal-extrapolation2}, consider the subfamily of $\mathcal{F}_\epsilon$ of functions of the form $f(z)=\hat{R}(z)+\alpha \tau z^n$, with $|\alpha|=1$ and $\tau=\epsilon \hat{R}(z_0)z_0^{-n}/(1+\epsilon)$ (included in $\mathcal{F}_\epsilon$ for small enough $\epsilon$), for which we have
 
 $$\frac{|f(z_0)-R|}{|f(z_0)|} = \frac{|\hat{R}(z_0)+\alpha\tau z_0^n-R|}{|\hat{R}(z_0)+\alpha \tau z_0^n|} \ge  \frac{|\hat{R}(z_0)+\alpha\tau z_0^n-R|}{|\hat{R}(z_0)|+| \tau z_0^n|}$$
 which, for $\alpha$ such that $\arg (\alpha\tau z_0^n)=\arg (\hat{R}(z_0)-R)$ is greater than
 $$ \frac{|\hat{R}(z_0)-R|+|\tau z_0^n|}{|\hat{R}(z_0)|+| \tau z_0^n|} \ge \frac{|\tau z_0^n|}{|\hat{R}(z_0)|+| \tau z_0^n|}=\frac{\displaystyle \frac{\epsilon\hat{R}(z_0)}{1+\epsilon}}{|\hat{R}(z_0)|\left(
     1+\frac{\epsilon}{1+\epsilon}  \right)} =  \frac \epsilon{1+2\epsilon}$$
 Combined with \eqref{isepsi}, this implies \eqref{eq:optimal-extrapolation2}.

For Part 3., we proceed in the same way as in Part 1., again mapping the questions to $\DD$. Using \eqref{eq:CauchyFormula} we get
\begin{equation}\label{eq:estiweight}
\frac{1}{\|f\|_W}\frac1{2\pi}\left|\oint \frac{s^{-n}f(s)ds}{s-z_0}\right| \le \inf_{r\in(|z_0|,1)}\frac{1}{r^{n-1}(r-|z_0|)W(r)}
\end{equation}
   where the contour of integration is $\partial \mathbb{D}_{r}(0)$ with $r\in (|z_0|,1)$. This proves  \eqref{eq:generalestimate}.

For \eqref{eq:optimal-extrapolation3} we proceed as in the proof of \eqref{eq:optimal-extrapolation} and  in the limit $\lambda\to \infty$ we obtain 
 $$ \frac{|z_0|^n}{\|z^n\|_{W}}=\frac{|z_0|^n}{\sup_{r\in [0,1)} r^{n} W(r)^{-1}}=|z_0|^n\inf_{r\in [0,1)}\frac{1}{W(r)r^{n}}$$

\end{proof} 
This is the optimal approximation that can be obtained based on this general information at some point $\omega_0\in\Omega$, as a function of $\omega_0$, and of the number $n$ of input Maclaurin coefficients.

\begin{Note}[All useful information in $P_n\circ\phi$ is contained in $(P_n\circ\phi)_n$.]\label{NExpGrowth}  
  {\rm  
    Restricting the analysis to a neighborhood of zero, we explain why,  in order to achieve optimality,  it is essential to discard information (by truncating  $P_n\circ\phi$ to $(P_n\circ\phi)_n$). The general case  follows from the proof of Theorem \ref{T1}. 

      We claim that any further terms of the Maclaurin series, as {\bf calculated} from $P_n\circ\phi$, lead to {\bf loss} of accuracy, in fact at a rate growing exponentially in $n$.
Indeed, assume $f\in \mathcal{F}_{P_n,W}$ and $\frac{1}{n!}f^{(n)}(0)=a_n$. A straightforward calculation shows that 
$$\frac{1}{n!}\left[(f\circ\phi)^{(n)}(0)-(f\circ\phi)_n^{(n)}(0)\right]=a_n c_1^n$$
     where we recall from Definition \ref{def:ck} that $c_1=\phi^\prime(0)$.
 As a consequence of Note \ref{LSchwarzLemma} and the normalization in Definition \ref{LD1} we have $c_1>1$ and hence, generically, $(f\circ\phi)_n^{(n)}(0)$ is ({\bf exponentially} in $n$) inaccurate. In other words anything beyond $(P_n\circ\phi)_n$ must, in general, be {\bf discarded}.

 In a deeper sense, the gain in accuracy is explained by the fact that the re-expansion is a representation in terms of functions for which the Cauchy-like contour {\em can be usefully pushed to the boundary of the whole Riemann surface}.
}\end{Note}

\begin{Note} Series coefficients extrapolation
\label{NCoeffExtrap}
{\rm In fact, this truncation also enables high accuracy extrapolation of the series expansion coefficients: see Section \ref{Scoefextrap}.}

  \end{Note}

\begin{Note}\label{N:02}
	{\rm    Optimality in Theorem \ref{T1} is relative to all functions analytic on $\Omega$ in a given class of bounds $W$.
		More detailed knowledge on the function to be reconstructed, as is available for resurgent functions ({\em e.g.} the concrete nature of their singularities, their behavior towards infinity along special curves in $\Omega$), can lead to significant further improvements: see Sections  \ref{sec:elim} and  \ref{S:empirical}. 
		To use the additional structure of the \'Ecalle alien calculus, one could use for instance his ``alien Taylor'' expansions described in \cite{Ecalle}, and the resurgence polynomials described in \cite{ETwisted}.}

\end{Note}

\begin{Note} Dependence on the class of bounds.
 \label{NBounds}
 {\rm  In some cases, {\em e.g.  the tritronqu\'ees of $P_{\rm I}$}, $W$ is significantly worse for the whole $\Omega$ than for a finite subcollection of its sheets. For $P_{\rm I}$ , $W$ is algebraic on finitely many sheets  whereas at infinity in $\Omega$ there is exponential growth, see Note \ref{unifp1}, \ref{N23ii}. Not far from zero, or if $n$ is large enough, complete uniformization is always optimal; indeed at any given $\omega_0$ the $W$ can influence at worst through a constant as seen on the left side of \eqref{eq:estiweight}. {\em If, however, $|z_0|$ is close to 1 and the degree of $P_n$ is not large enough}, then $W$  starts to matter, as seen in Theorem \ref{T1}.
 Thus, close to $\partial\DD$ with a "small" $n$, it may be beneficial to cut out whole regions of $\Omega$ from the analysis, if they are not needed, to improve $W$. One can use the formulas in the theorem to optimize over $n, z_0,W$ as well. 
			
			A specific way to truncate the Riemann surface is {\em partial uniformization} explained in \S\ref{S:partialU}.
		
}\end{Note}

\begin{Note} \label{logcap} Connection with logarithmic capacity. {\rm As mentioned in Note \ref{LSchwarzLemma}, the acceleration modulus $|\psi(\omega)/\omega|<1$ if $\Omega \supsetneq\DD$, and there is always improvement of accuracy. It clearly also follows that $|\psi'(0)|$, the acceleration modulus near zero is also $<1$. If $\Omega$ is a domain in $\CC$, this latter quantity has an interesting geometric interpretation:
		
Let	 $\omega_0\not\in\Omega$ and $K=1/(\Omega-\omega_0)$ the  reciprocal of the shifted domain (as defined in \S\ref{sec:summary}). If $K$ is compact, then the constant $|\psi'(0)|$ is the {\em logarithmic capacity (or trans-finite diameter) of $K$}: see \cite{Landkof,Ransford,Saff} and \S\ref{S51}-\S\ref{S52}.
}\end{Note}

\section{Uniformization of Riemann surfaces of the Borel plane  for ODEs}
\label{sec:borel-odes}

Uniformizing Riemann surfaces is a non-trivial problem, and  is generally not explicit. However, it turns out that the Riemann surface $\Omega$  can be uniformized in closed-form for the Borel plane of the solution to a large class of second order ODEs, both linear and nonlinear. This class includes the Painlev\'e equations P$_{\rm I}$-P$_{\rm V}$. Closed-form uniformization can also be achieved for higher-order ODEs having sufficient symmetry.

\subsection{Uniformization of the Riemann Surface for Linear Second Order ODEs}
\label{sec:odes}

 We begin by describing a particular uniformizing  conformal map (\ref{eq:Nome1})--(\ref{eq:inome0}) expressed in terms of the elliptic nome function.\footnote{Some properties of this map, other than its uniformizing features, are described in \cite{Nehari}, p. 323.} This  map is important for our analysis because it is used to construct the new uniformization of the Borel plane for Painlev\'e solutions, as discussed in section, \S\ref{sec:P1-unif}.
 In the next section, \S\ref{sec:composition}, we also present a new explicit construction of this map in terms of a rapidly convergent iterative composition of elementary conformal maps.

Consider  $\Omega(\hat\CC\setminus \{(0),1,\infty\})$ (see Note \ref{NoteN1} and Definition \ref{DOmega()}).
Then the conformal map $\psi:\Omega\to\DD$ is the {\em elliptic nome} function $q$:
\begin{equation}
  \label{eq:Nome1}
  z= \psi(\omega)=q(\omega)=e^{-\pi\frac{\mathbb{K}(1-\omega)}{\mathbb K(\omega)}}
\end{equation}
and  $\omega=\phi(z)$, where $\phi$ is the {\em inverse elliptic nome} function. We have for small $z$
\begin{equation}
  \label{eq:inome0}
 \omega= \phi(z)= 16 z-128 z^2+704 z^3+\cdots
\end{equation}

  \begin{proof} 
 We start with the well known uniformization over the upper half plane $\mathbb H$ of  $\Omega\left(\hat \CC \setminus\{0,1,\infty\}\right)$ by $\tau(\omega)= i\mathbb{K}(1-\omega)/\mathbb K(\omega)$, see \cite{Bateman}, p 99. Note that $q=e^{\pi i \tau}$.  The map $\tau$ is also conformal between $\mathbb H$ and the geodesic triangle $\Delta$, see Fig. \ref{FNome}. The inverse function,  $\lambda =\tau^{-1}$ is the modular elliptic function.
  \begin{figure}[h!]
  \centering
  \includegraphics[scale=0.9]{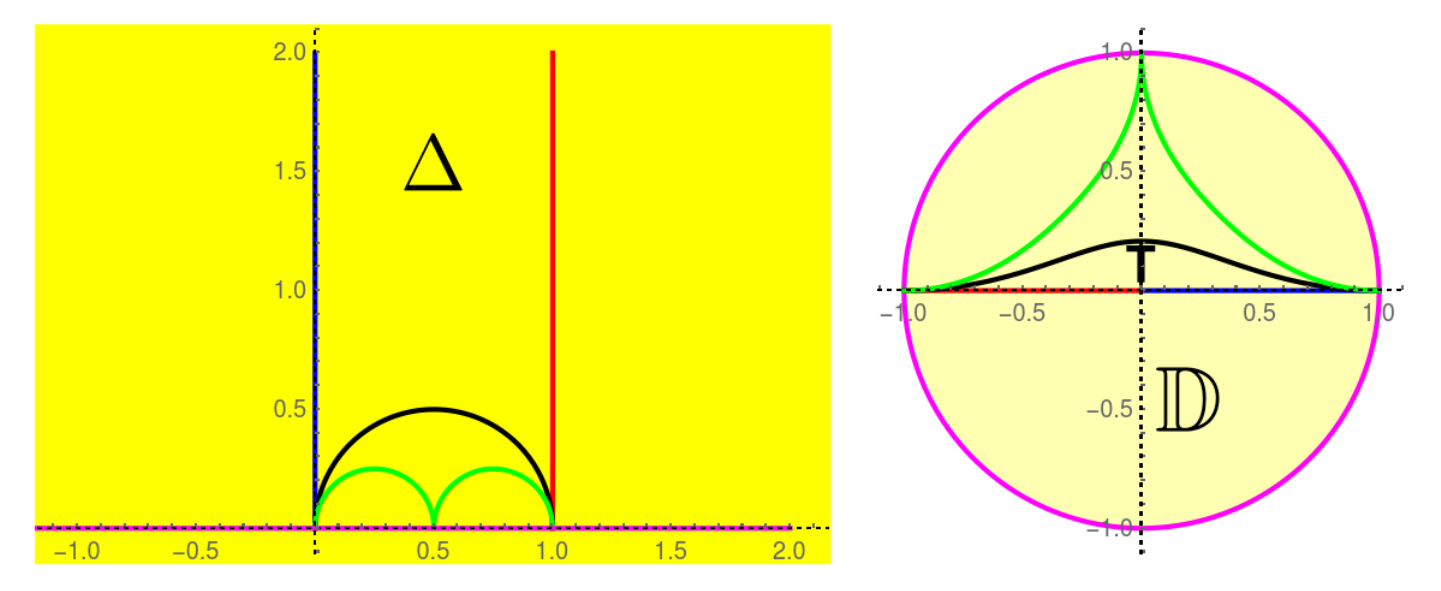}
   \caption{Images of geodesic triangles in the upper half plane $\mathbb{H}$ through $h(\omega)=\omega\mapsto e^{\pi i \tau(\omega)}$. 
   Like-colored pieces of the boundaries correspond to each other through $h$. }
   \label{FNome}
  \end{figure}
  The set  $T=e^{\pi i \tau(\Delta)}$ is the curvilinear triangle $T$ in Figure \ref{FNome},  where the upper side  is an analytic curve.  We aim to show that $ W=q^{-1}$ extends analytically to $\DD$ by successive Schwarz reflections of $T$, and its reflections, across their sides. We also note that with $\omega\in T$ we have $W(z)=\lambda(\frac{1}{\pi i}\log z)$. Since $\lambda$  is analytic in the upper half plane, $W$ admits analytic continuation through the rays above. Using the periodicity property $\lambda(t+2)=\lambda(t)$ we see that the monodromy of $W$ around zero is trivial. Furthermore, the product representation of $\lambda$ for $\Im \omega>0$,
   \begin{figure}[ht!]
  \centering
  \includegraphics[scale=0.7]{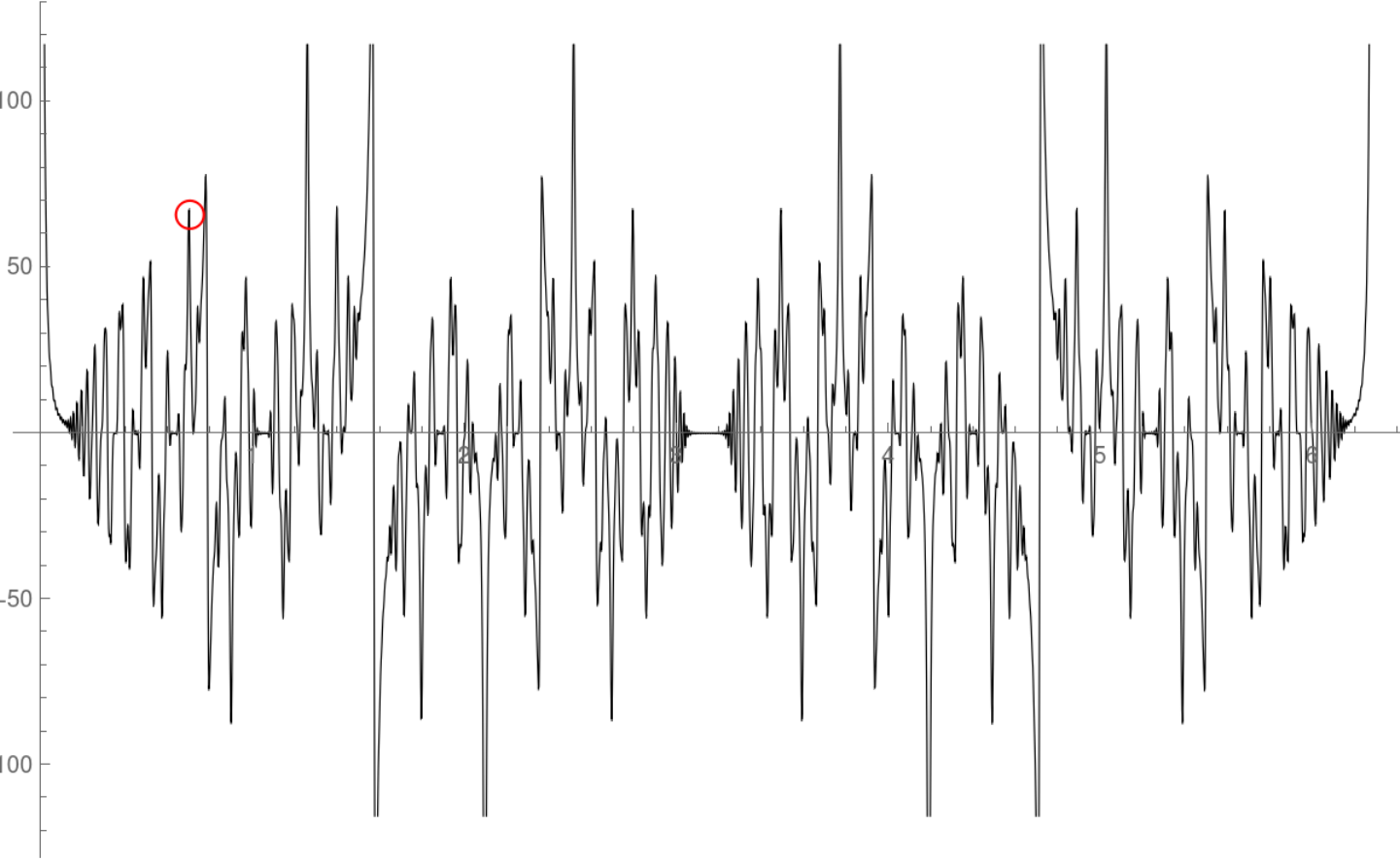}
   \caption{A plot of $|\mathbb K\circ\phi|$ on a circle of radius $0.997$ using  200 nonzero terms of its Maclaurin series. See Note \ref{N:nome}. }
\label{fig:Ksurface}
\end{figure}
  $$ \lambda(\omega)=16 e^{\pi i \omega}\mathlarger{\prod}_{k=1}^\infty\left(\frac{1+e^{2k \pi i \omega}}{1+e^{(2k-1) \pi i \omega}}\right)^8$$
  shows that $\lim_{t\to i\infty}\lambda(t)=0$, hence $z=0$ is a point of analyticity of $W$, and $W(0)=0$. Each Schwarz reflection of $T$ mirrors Schwarz reflections of $\Delta$. It is clear that the set of all these reflections cover $\DD$ since $e^{\pi i \mathbb H}=\DD$ and, as in the uniformization proof for $\lambda$, the reflections of $\Delta$ cover $\mathbb H$. The curves starting at zero on $\Omega\left(\hat \CC\setminus \{0,1,\infty\}\right)$ correspond to curves in $\DD$. 
  
   For a related construction, see also \S\ref{sec:composition}.
  \end{proof}

\begin{Note}\label{N:nome}{\rm 

\begin{enumerate}
\item  As an example, Fig. \ref{fig:Ksurface} shows the singularities on the Riemann surface of the complete elliptic integral $\mathbb K(\omega)$, seen after uniformization through $\phi$ in (\ref{eq:inome0}), using $(P_{200}\circ\phi)_{200}$. 
In the figure, the marked singularity is at $0$, after a clockwise loop around $1$ and a clockwise loop around $\infty$.  The higher the Riemann sheet, the more suppressed is the singularity. Properly dilated, the singularities exhibit a periodic structure, reflecting the simple monodromy group of $\mathbb K$. Compare with Fig. \ref{fig:PIsurface}, the Riemann surface of the Borel transform of the  tronqu\'ee of P$_{\rm I}$ which has infinitely many singularities on each Riemann sheet, and exhibits a less regular structure.

\item We note that the only possible singularities of $\mathbb K$ on any Riemann sheet are $\{0,1,\infty\}$. The number of visible sheets of $\Omega$ is roughly $\frac13$ of the number of visible peaks.
\end{enumerate}
}
\end{Note}

\begin{Note}\label{N:nome-number} {\rm The Maclaurin series of the function $\mathbb K\circ\phi$
		has interesting number theoretical properties. It is a lacunary series whose powers with nonzero coefficients are the numbers that are the sum of two squares, cf. OEIS A001481, and these coefficients are $\frac{\pi}{2} a_k$ where the $a_k$s are the  number of ways of writing $k$ as a sum of at most two nonzero
		squares where order matters, cf. OEIS A002654.		
	}
\end{Note}

\subsection{Uniformization by composition of elementary conformal maps}
\label{sec:composition}

In this section we introduce a new way to uniformize non-trivial Riemann surfaces using elementary maps. As a particular example we show that the Riemann surface of the Borel plane for second order ODEs  can be uniformized by  a limit of compositions of elementary conformal maps. Thus instead of using the transcendental  inverse elliptic nome function, as in  (\ref{eq:Nome1})--(\ref{eq:inome0}), the uniformizing map can be well approximated by elementary functions. This is, in fact, how we first obtained the map $\psi$ given in \S\ref{sec:odes}.

The essential idea here is to open up, after $n$ compositions,  $n$ Riemann sheets of a Riemann surface. We demonstrate this procedure with an example.  Let $\psi_1$ and $\phi_1$ be the  elementary conformal maps for $\Omega=\CC \setminus [1, \infty)$, (see Eq. (\ref{eq:r1r2}) in Example \ref{mini12} of Note \ref{N:sing1}, also discussed in the Appendix, Eq. \eqref{eq:onecutmap}), which map a one-cut plane into the unit disk. Further, define $\psi_n(\omega)=[\psi_1(\omega^n)]^{1/n}$, with inverse $\phi_n(z)=[\phi_1(z^n)]^{1/n}$, for $n \in \mathbb Z$, with the usual branch choices.
\begin{Theorem}\label{Th:compo}
  The composition map $\psi_{2^n}\circ\psi_{2^{n-1}}\circ\cdots \circ \psi_1(\omega)$ converges, as $n\to\infty$,   to the uniformization map $\psi(\omega)=e^{-\pi \, \mathbb K(1-\omega)/\mathbb K(\omega)}$ in \eqref{eq:Nome1}, the elliptic nome function.
\end{Theorem}
\begin{proof}
  We first note that the singular points of $\psi_k$ are $\{0,\{\sigma_j\}_{j\le k-1},\infty\}$,  where $\{\sigma_j\}_{j\le k-1}$ are the $k$-th roots of unity (with zero a point of analyticity on the first Riemann sheet). Secondly,  for any $n\in\NN$,  $\psi_{2^n}\circ\psi_{2^{n-1}}\circ\cdots\psi_1$ is only singular at $\{0,1,\infty\}$ (with zero a point of analyticity on the first Riemann sheet); this follows immediately by examining the singularities of $\phi_k$ for $k\in\NN$. Convergence of the composition follows from the fact that $\psi_k=2^{-2/k}\omega(1+O(\omega^k))$ (see also the proof of Lemma \ref{L29}).

  Next, we note that the non-constant term of the Puiseux expansion of $\psi_n$ at the finite  nonzero singularities is of the form $\sigma_j+const \sqrt{\omega-\sigma_j}(1+o(1))$, and $1+const. \omega^{-1/2}(1+o(1))$ at infinity. We describe  the monodromy group of $\psi_k$ in terms of the  generators $r_{\sigma_j}$ and $r_{\infty}$,  the local monodromies at $\sigma_j,\infty$. A straightforward calculation shows that any of the elements $r_{\sigma_j}^2$, $r_{\infty}^2$,  $r_\infty r_{\sigma_j}$ of the monodromy group of the Riemann surface of $\psi_k$ is mapped on a single $r_{\sigma_j}$ or $r_{\infty}$  of the monodromy group of the Riemann surface of $\psi_{2^k}$, while $r_\infty r_{\sigma_j}$ for $i\ne j$ is mapped inside $\DD$. Injectivity of the limit map is also straightforward.

An alternative proof, based on convergence  of $\phi_1\circ \phi_2\circ\cdots$, is given in Lemma \ref{L29}:
\end{proof}

 \begin{Lemma}\label{L29}
  \label{Literation}
  \begin{enumerate}
  \item We have 
  \begin{equation}
    \label{eq:infinitecompo-1}
   \phi_1\circ \phi_2\circ \phi_{4}\circ\cdots\circ \phi_{2^k} \to q^{-1}\quad, \quad k\to\infty
  \end{equation}
uniformly in $\DD$,  where $q^{-1}$ is the inverse elliptic nome function.
  
  \item Moreover, starting the doubling iteration with $\phi_n$ instead of $\phi_1$, we have more generally
    \begin{equation}
    \label{eq:infinitecompo-n}
    \phi_n\circ \phi_{2n}\circ \phi_{4n}\circ\cdots\circ \phi_{2^k n}\to z\mapsto [q^{-1}(z^n)]^{1/n}  \quad, \quad k\to\infty
  \end{equation}
  which uniformizes   $\Omega\left(\hat\CC \setminus S\right)$, where $S$ are the $n$-th roots of unity. 
 
  \end{enumerate}
 \end{Lemma}
 
\begin{proof}
   We note that $z(\omega)$ is analytic in the unit disk (see also \eqref{eq:inome0}). 

   The Landen transformation \cite{Bateman} implies that the map $\phi_1(\omega)$ satisfies  $\phi_1(w(z^2)^{1/2})=w(z)$, and therefore $w(z^2)^{1/2}=\phi_1^{-1}(w(z))$.
Iterating this identity, we find in general, for $m\in\NN$ and $z\in\DD$, that 
$$w(z^{2^{m+1}})^{1/2^{m+1}}=\phi^{-1}_{2^m}\circ\cdots\circ \phi_1^{-1}(w(z))$$
For any function analytic within the unit disk and such that $f(z)=cz(1+o(z))$, and for any $r<1$  we have $\lim_{n\to \infty}f(z^n)^{1/n}=z$ uniformly in $\DD_r(0)$ (since $f(z)/z$ is uniformly bounded there). This implies
$$\lim_{m\to \infty} \phi^{-1}_{2^m}\circ\cdots\circ \phi_1^{-1}(w(z))=z$$
The result (\ref{eq:infinitecompo-1}) follows by straightforward function inversions. The proof of (\ref{eq:infinitecompo-n}) for general $n$ is very similar.
The proof of (\ref{eq:landen}) follows by noting that $\phi_n(\omega)=\phi_1(\omega^n)^{1/n}$. 
\end{proof}
\begin{Note}{\rm 

\begin{enumerate}

\item

  The limit in Lemma \ref{L29} can also be expressed as an infinite iteration limit of the ascending  Landen transformation, $L(z)=2\sqrt{z}/(1+z)$:
  \begin{equation}
  \lim_{k\to\infty} L\circ L \circ \dots \circ L(z^{2^k})=\sqrt{q^{-1}(z^2)}
  \label{eq:landen}
  \end{equation}

\item
This result has an important  practical implication: we can  approximate a complicated (e.g., elliptic) map by a composition of elementary (e.g., rational) maps.  For example,  $\phi_1\circ\phi_2\circ\circ \dots \circ \phi_{64}$ ``opens up'' 6 Riemann sheets and has an acceleration modulus of 15.8 at $\omega=0$ (compared with 16, for actual uniformization) and distorts the point $\omega=1-10^{-33}$ to $0.9$ in the conformal disk, compared to the ideal $1-10^{-40}$ of the limit map. 

\item
The results of the following section show that these composition maps can also be used to provide accurate approximate uniformizations of the Riemann surface of the Borel transform of solutions to {\bf nonlinear} second order  ODEs.

\end{enumerate}

}\end{Note}
\begin{Definition}
Let  $\Omega_{\ZZ}$ be the set of equivalence classes of curves starting at $0$, modulo homotopies in $\CC\setminus \ZZ$\footnote{Not $\hat{\CC}\setminus\ZZ\setminus\{\infty\}$; a curve around infinity is undefined, since $\ZZ$ is not compact.}. 
\end{Definition}
This is the Riemann surface of analyticity of solutions of linear or nonlinear ODEs with eigenvalues normalizable to $\pm 1$ \cite{Duke}, such as tronqu\'ee solutions of $P_{\rm I}-P_{\rm V}$. \footnote{ For some of these ODE solutions, {\em e.g.} the tronqu\'ees of $P_{\rm I}$, the Riemann surface $\Omega$ is "larger", and not simply connected (all odd integers are square root branch points); we pass instead to the universal cover of $\Omega$ which is indeed $\Omega_{\ZZ}$. See also footnote \ref{ftn4} on p. \pageref{ftn4}.}
\subsection{Uniformization of the Riemann surface of  $\Omega_{\ZZ}$}
Figure \ref{fig:PIsurface} shows the singularity structure of $Y$, the Borel transform of the tritronqu\'ee solution $y$ of Painlev\'e P$_{\rm I}$, on its Riemann surface. Starting from a finite number of terms of the asymptotic expansion, the uniformization of the associated Riemann surface described in Theorem \ref{T:8} permits analytic continuation onto  the higher sheets of this Riemann surface.
\label{sec:P1-unif}
 \begin{figure}[h!]
   \centering
  \includegraphics[scale=0.4]{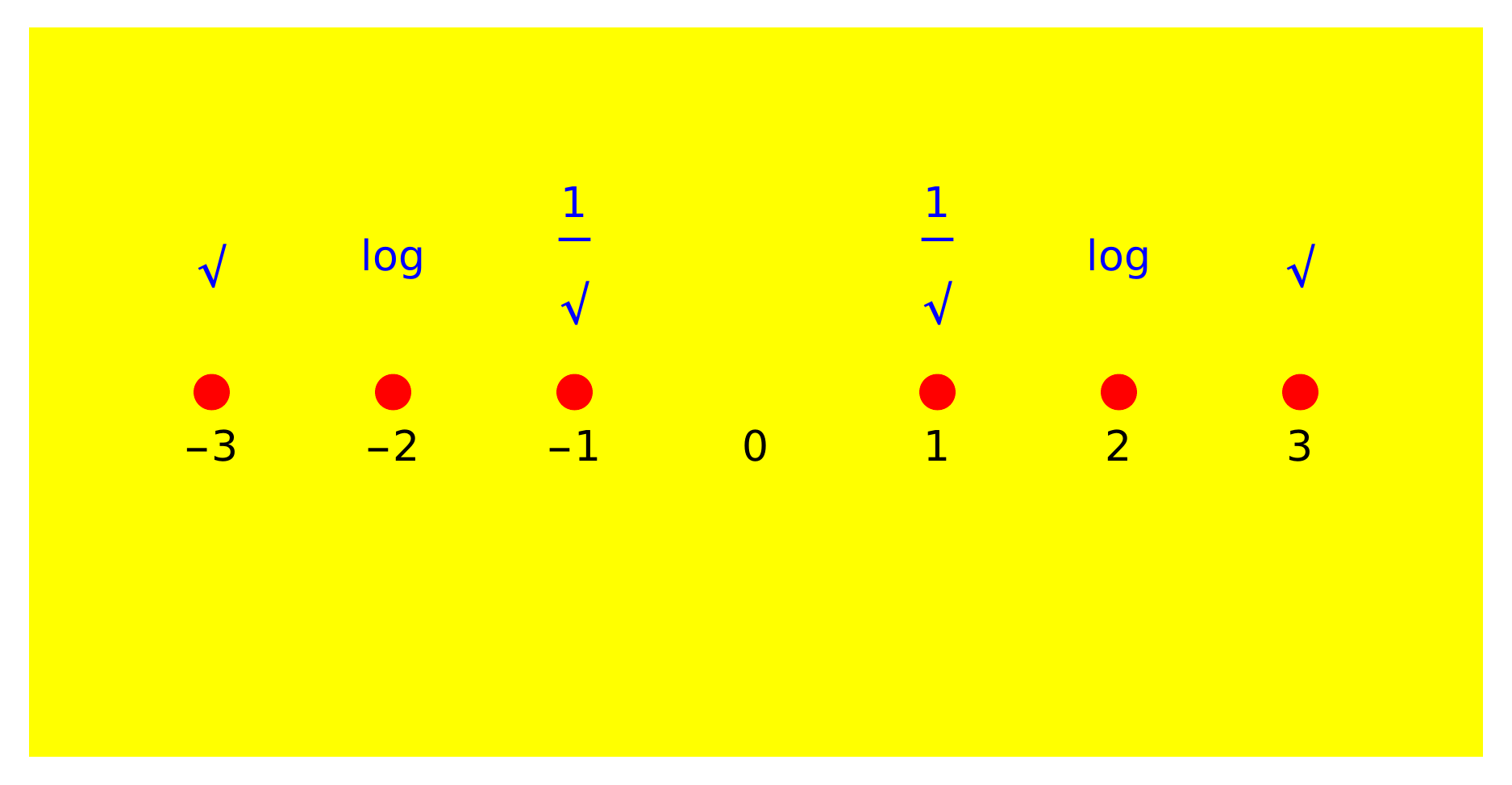}
   \caption{The singularities of the Borel transform of the {\it tronqu\'ee} solution of the Painlev\'e equation P1, which occur at all integers, except for zero on the  first Riemann sheet (see Note \ref{NoteN1}). The type of singularity is indicated in blue. Here, for instance, $"\sqrt{\ }"$ means a local structure of the form $\sqrt{p\mp n}A(p)+B(p)$, with $A,B$ locally analytic. The pattern is that the singularities weaken by a factor $(p \mp n)^{|n|/2}.$}
\label{fig:P1singularities}
\end{figure}
\begin{figure}[h!]
  \centering
\raisebox{8cm}{\scalebox{1}{$|Y(\phi(e^{i\theta)}))|$}}\hspace{-0cm}  \includegraphics[scale=0.8]{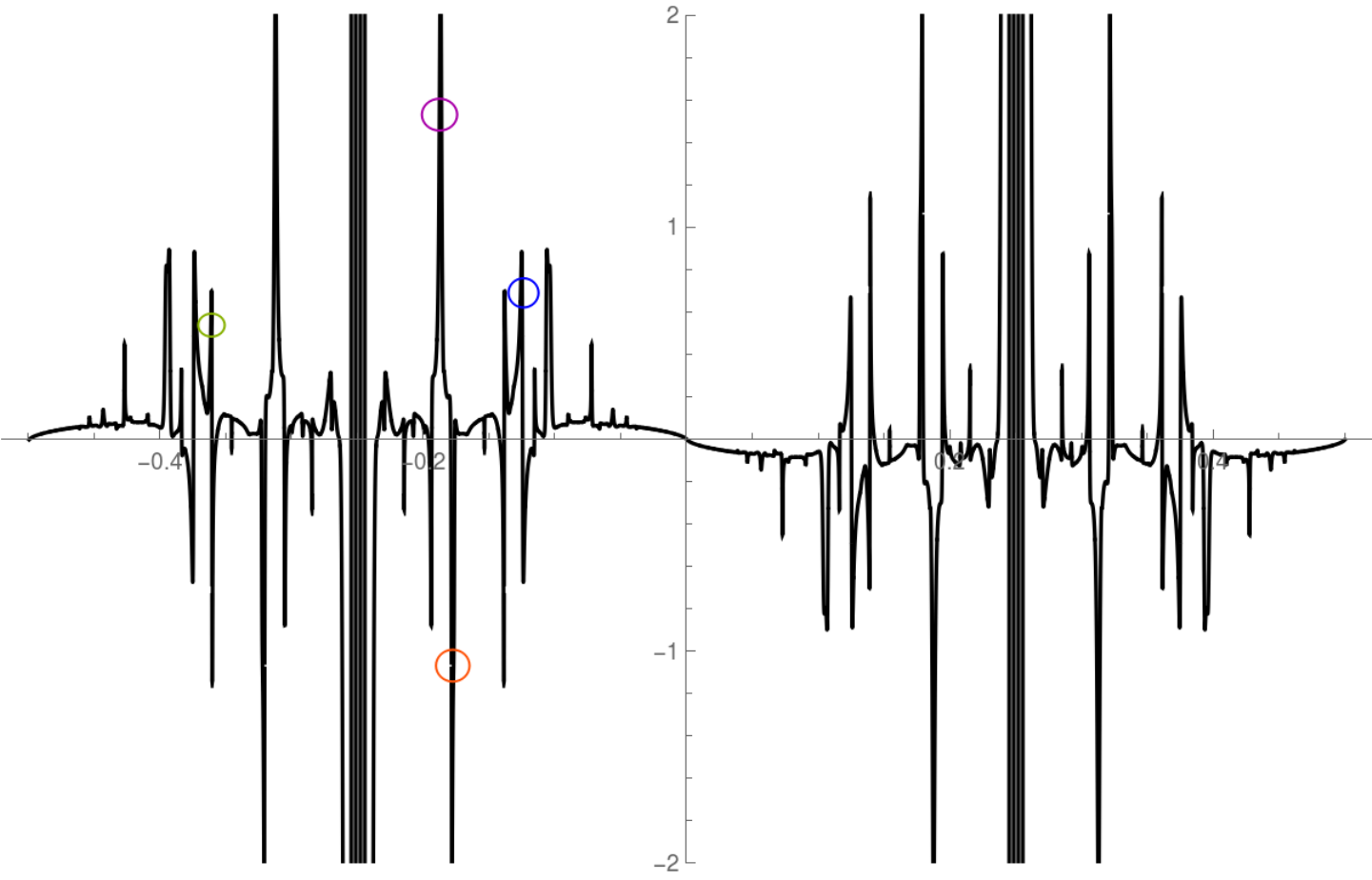}\raisebox{4cm}{$\frac{\theta}{2\pi}$}
   \caption{Singularities of  $Y\circ\phi$ along the conformal circle. See Note \ref{unifp1}.}
\label{fig:PIsurface}
\end{figure}
\begin{Theorem}\label{T:8} 
  $\Omega_{\ZZ}$ is uniformized by $\psi=\phi^{-1}$, where $\phi= \frac{1}{2\pi i}\ln(1-q^{-1})$, with $q$ the elliptic nome function, as in \eqref{eq:Nome1}.
 \label{T:P1}
\end{Theorem}

\begin{proof}
 This result follows from the map in  section \S\ref{sec:odes}
 (see equations (\ref{eq:Nome1}) and (\ref{eq:inome0})), and the following lemma.
  \begin{Lemma} 
The function $\Phi=\omega\mapsto \frac1{2\pi i}\ln(1-\omega)$ maps conformally $\Omega(\hat\CC\setminus \{(0), 1, \infty\})$, discussed in \S\ref{sec:odes},
  onto  the Riemann surface $\Omega_{\ZZ}$, defined above as the set of equivalence classes of curves starting at $0$, modulo homotopies in $\CC\setminus \ZZ$.
\end{Lemma}
\begin{proof}
Let $\mathcal G$ be the free group with generators $r_k$, where for each $k\in\ZZ$, the generator $r_k$ is the anticlockwise rotation around $k$.   An equivalence class of curves of  $\Omega_\ZZ$ can be described by a word $r_{k_1}\cdots r_{k_m}$, for some $m$ plus a piecewise linear arc in $\CC$. Let $\tilde{r}_0,\tilde{r_1}$ the generators of the group of $\Omega$. The result follows by noting that the element $r_k\in\mathcal{G}$ is obtained through $\Phi$ from $\tilde{r}_0\tilde{r_1}^k$. Injectivity is clear. (In words, through $\Phi$, any ``address'' in $\Omega_{\ZZ}$ can be reached uniquely from an address on $\Omega$.)
\end{proof}
\end{proof}

\begin{Note} \label{unifp1}{\rm 
    \begin{enumerate}
       \item Let $P_{200}$ denote the Maclaurin polynomial with 200 nonzero coefficients of the Borel transform $Y$ of the tronqu\'ee solutions $y$ of Painlev\'e P$_{\rm I}$. $|(P_{200}\circ\phi)_{200}|$ is plotted in Fig.\ref{fig:PIsurface} at $0.1$ distance from the boundary  $\partial\DD$. Here $\phi$ is the uniformizing map given in Theorem \ref{T:P1}. The circles correspond to singularities of $Y$ at $-2$, reached by analytic continuation from below (green) $-1$ after analytic continuation around $+1$ (magenta), $0$ after analytic continuation around $+1$ (red) and $-2$ reached from above.

  \item \label{N23ii}  \label{allsheets} 
 Importantly, $Y$ does not have exponential growth on any particular Riemann sheet  but there is substantial strengthening of the singularities at $\pm n$ on the $n$-th sheet, \cite{Duke}, which results in cumulative exponential growth on $\Omega$ (the thick lines in the figure). The growth of the odd coefficients of $(P_{200}\circ\phi)_{200}$ matches $\sim - i (-1)^k e^{1.73 \sqrt{k}}, k=2j+1$ which implies exponential growth towards the boundary of the conformal disk, with blow-up rate roughly $\exp[\frac34\text{dist}(z,\DD)^{-1}]$.  In this sense $(P_{200}\circ\phi)_{200}$ can see the totality of the Riemann surface.

\item \label{difwght} Because exponential growth only occurs in exceptional directions in $\Omega$, not close to the origin, and for $n$ not large enough, uniformizing the whole of $\Omega$ is suboptimal: a further conformal map of $\DD\setminus S$, where $S$ is a pair of small circles around the exponential singularities would take the conformal series from a class with exponential bounds to one with polynomial bounds, improving the accuracy everywhere except close to the exponential singularities where the information would be obstructed by $S$. In Figure \ref{fig:PIsurface} we used instead (non-rigorously) Pad\'e approximation, which is directionally sensitive.  

  \end{enumerate}

}\end{Note}

\subsubsection{Partial uniformization}\label{S:partialU} If one uses a finite composition $ \phi_1\circ \phi_2\circ \phi_{4}\circ\cdots\circ \phi_{2^k} $ instead of $\phi=q^{-1}$ in \eqref{eq:infinitecompo-1}, the result is a truncation of the full Riemann surface $\Omega$: only a finite number $k$  of sheets of $\Omega$ are opened up. For functions living on $\Omega_{\ZZ}$, partial uniformization is achieved using $\tilde{\phi}=\frac{1}{2\pi i}\ln(1-\phi_1\circ \phi_2\circ \phi_{4}\circ\cdots\circ \phi_{2^k})$. With $P_n$ of small degree, for a $z_0$ close enough to $\partial \DD$, and if the weight $W$ is significantly worse for $\Omega$ than for its truncations, further gain can be achieved by using the bounds in Theorem \ref{T1} to optimize over the triple $N,z_0,W(N)$ as explained in Note \ref{NBounds}.

Another use of partial uniformization is in analyzing the local behavior near singularities when more information is available, see Note \ref{NPartial}.

\subsection{Illustration of the gain in accuracy} \label{SgainInAcc}
\label{sec:t1discussion}

\subsubsection{Improvement of accuracy near zero} 
Accelerating the convergence, even near zero, is important when only a finite number of terms of the series is known, and our goal is to be as accurate  as possible in the reconstruction. Consider first the domain $\Omega_1=\CC\setminus [1,\infty)$, and let $F$ be analytic in $\Omega_1$ and bounded by 1. Then,
\begin{eqnarray}
	|F(\omega)-P_n(\omega)|\sim  |\omega|^n;\qquad n\to \infty;\qquad  |\omega|\text{ small}
	\label{eq:naive}
\end{eqnarray}

The procedure in Theorem \ref{T1} yields instead
\begin{equation}\label{oeq}
	|F(\omega)-(F\circ\phi)_n\circ\psi(\omega)|\sim 4^{-n}|\omega|^n;\qquad n\to \infty ;\qquad  |\omega|\text{ small}
\end{equation}

If  however, as it is often the case, $[1,\infty)$ is not a natural boundary of $F$ and,  instead, $F$ simply has branch point at $1$, say  $F$ is analytic on $\Omega\left(\hat\CC\setminus \{(0),1,\infty\}\right)$, then, using the maps obtained \eqref{eq:inome0} and \eqref{eq:Nome1} in \S\ref{sec:borel-odes} we get
\begin{equation}\label{oeq2}
	|F(\omega)-(F\circ\phi)_n\circ\psi(\omega)|\sim 16^{-n}|\omega|^n;\qquad n\to \infty ;\qquad  |\omega|\text{ small}
\end{equation}
which is a significant gain in precision near the origin.

\subsubsection{Improvement of accuracy for larger values of $\omega$ and in approaching singularities}\label{Saccsing}

The accuracy improvement using the optimal procedure is particularly dramatic when approaching the boundary $\partial\Omega$, crucial  when probing singularities. If $\Omega=\Omega\left(\hat\CC\setminus \{(0),1,\infty\}\right)$, then $\psi(1-10^{-40})=0.9012\cdots$, 
meaning that, if $\hat{R}_n$ requires $n=50$ terms for an accuracy of $.5\%$, then $P_n$ would require, for the same accuracy,  $n\sim 10^{40}$ terms of the {\bf same} input series, a practical impossibility.

More generally, a similar improvement occurs when $\Omega=\Omega\left(\hat\CC\setminus S\right)$, where $S$ is a discrete set. Since $\psi$ has to accommodate the Riemann surface of a logarithm at the points in $S$, as $\omega$ approaches a puncture in $S$, the uniformization has a logarithmic singularity there. Therefore, the Euclidean distance in $\CC$ between a point $\omega$  to the puncture is exponentially smaller than the distance between $\psi(\omega)$ and the point on $\partial\DD$ corresponding to the puncture, resulting in an  exponential distance distortion, roughly as in the example of $\Omega=\Omega\left(\hat\CC\setminus \{(0),1,\infty\}\right)$ discussed above.

Even with very few terms, the accuracy gain allows for approaching singularities enough to find their nature.  Here we illustrate this using $P_9$ for the complete elliptic integral of the first kind $F(\omega)=\mathbb K(\omega)$,
which is analytic on $\Omega\left(\hat\CC\setminus \{(0),1,\infty\}\right)$.
Consider the analytic continuation of its Maclaurin polynomial $P_9$ (i.e., just 8 terms of the expansion about $\omega=0$).  Figure \ref{PlotEllipticK} plots the $\log_{10}$ error of the approximate reconstruction, as the singularity at $\omega=1$ is approached. The red curve is the truncated series $P_9(\omega)$ itself, which is clearly a very bad approximation near the singularity. The pink curve 
shows the maximal order near-diagonal Pad\'e approximant computable from $P_9$, which is an improvement, but still fails close to the singularity. The dashed-blue curve shows the optimal procedure $(P_9\circ\phi)_9\circ\psi$, which is much more accurate as the singularity is approached.
\begin{figure} [ht!]
	\centering
	\includegraphics[scale=0.9]{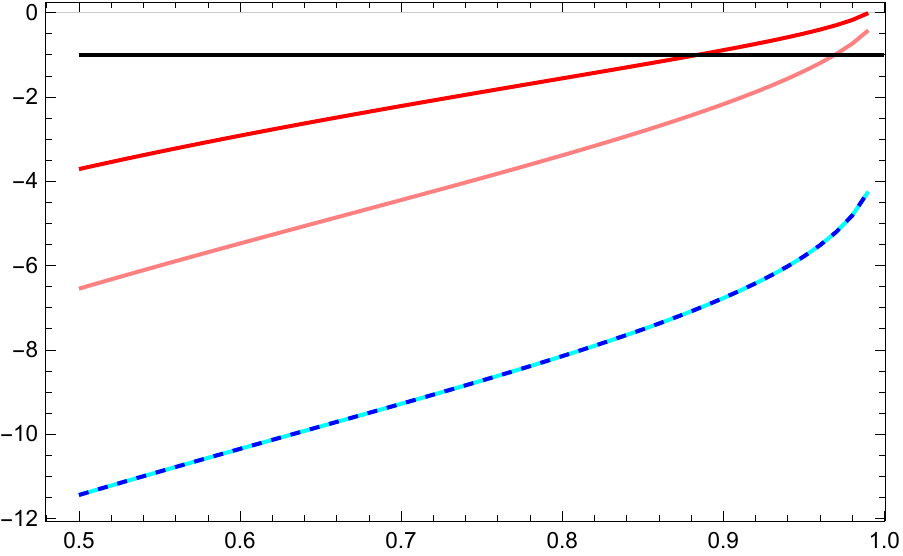}
	\caption{$\log_{10}$ of the error in calculating  $\mathbb K(\omega)$ on the interval $(0.5,0.99)$ directly from $P_9$ (red), from the maximal order near-diagonal Pad\'e approximant computable from $P_9$ (pink), 
		and from the optimal procedure $(P_9\circ\phi)_9\circ\psi$ (dashed-blue).
		This dashed-blue curve is indistinguishable on this scale from the series extrapolation with 472 extra terms, as explained in \S\ref{Scoefextrap}.
		The black line marks an accuracy of $0.1$.}	
	\label{PlotEllipticK}
\end{figure}

Of course, outside the unit disk the Maclaurin series diverges. In Figure \ref{PlotEllipticK-out} we compare the optimal procedure and Pad\'e, both based on $P_9$. We see that the optimal procedure is much more accurate  along the cut, beyond the radius of convergence, than the Pad\'e approximation.

\begin{figure}[ht!]
	\centering
	\includegraphics[scale=0.9]{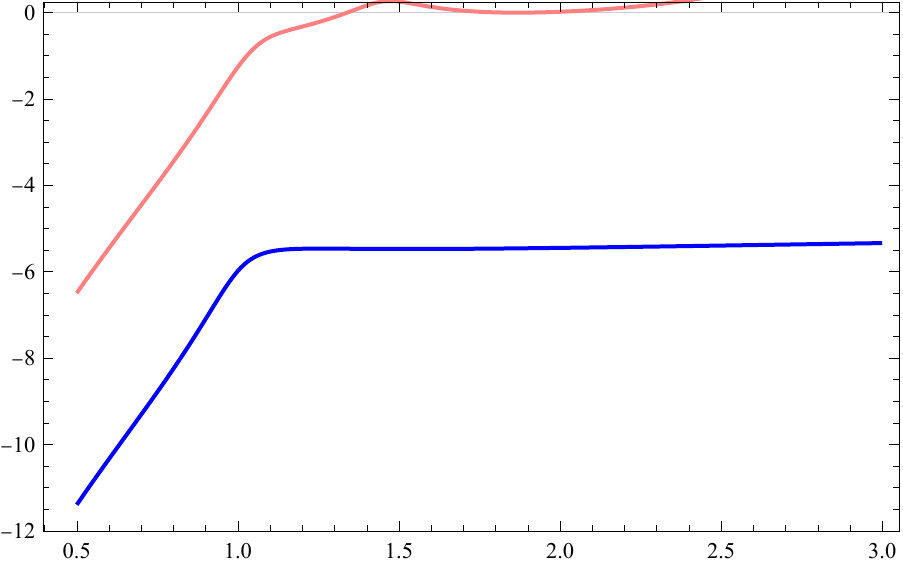}
	\caption{$\log_{10}$ of the error in calculating  $\mathbb K(\omega)$ from $P_9$ on the line $(1/2+i/10,3+i/10)$ using the optimal procedure (blue) and the maximal order near-diagonal Pad\'e approximant computable from $P_9$ (pink).  }	
	\label{PlotEllipticK-out}
\end{figure}

	\subsubsection{Singularities on other Riemann sheets}

Consider again the Riemann surface  $\Omega\left(\hat \CC\setminus \{(0),1,\infty\}\right)$, relevant for the elliptic integral function $\mathbb K(\omega)$. A loop starting at $0$ on the first Riemann sheet, rotating  clockwise around the singularity at $1$ and returning to $0$, now approaches a logarithmic singularity at $0$ on the second sheet. 
This is illustrated by the orange curve in 
Figure \ref{fig:Ksurface-loops}. We can also consider a curve that loops {\it twice} around the singularity at $1$ before returning to $0$, which is now a singularity on the third sheet. This is shown as the purple curve in Figure \ref{fig:Ksurface-loops}. A more complicated closed loop consists of starting at $0$ on the first sheet, looping around both $1$ and $0$ before approaching the singularity at $1$: this is shown as the magenta curve in Figure \ref{fig:Ksurface-loops}.

 \begin{figure}[ht!]
  \centering
  \includegraphics[scale=0.9]{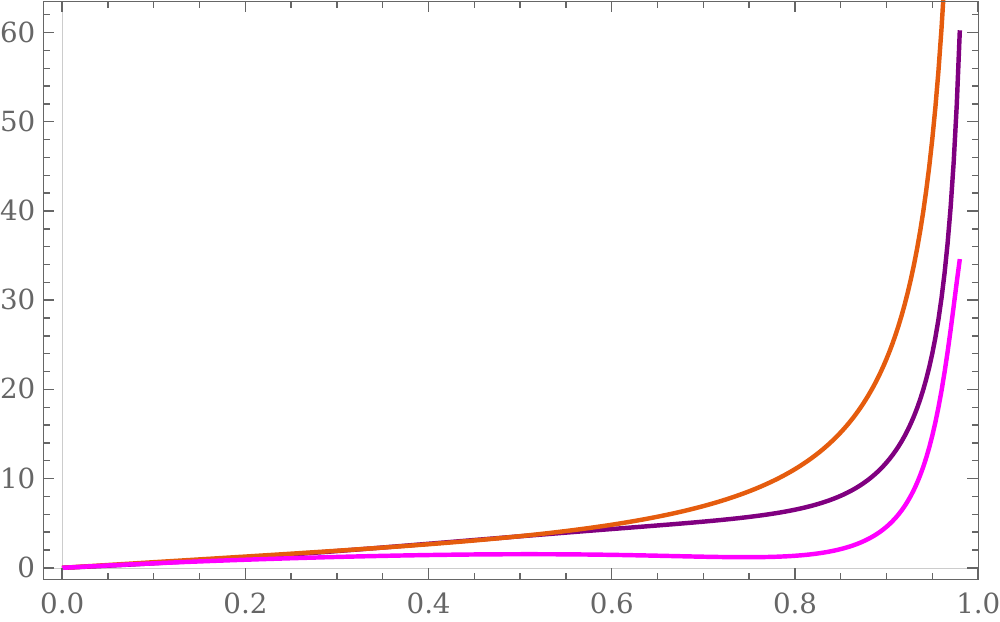}
   \caption{A plot of $\mathbb K\circ\phi$ on the line $\ell=\{i t:t\in [0,0.98]$ using  200 nonzero terms of its Maclaurin series; in orange $\phi(\ell)$ traverses the circle $\{\omega=1+e^{-i\tau}:\tau\in [0,2\pi-\epsilon]$ where $|\epsilon|\sim 10^{-52}$. 
   This shows the behavior of $\mathbb{K}$ very close to the logarithmic  singularity at $0$ on the second Riemann sheet. Purple: approaching the singularity at $0$ on the third Riemann sheet after looping twice around $1$. 
Magenta: approaching the singularity at $1$ after looping around  the singularities at $1$ on the first sheet  and $0$ on the second sheet.  Approaching a singularity at this remarkable distance allows for determining the {\em type} of singularity, and detailed information about the {\em local behavior} there.}

\label{fig:Ksurface-loops}
\end{figure}

\subsection{Accurate extrapolation of series coefficients}\label{Scoefextrap}

Let $Q_n$ be the polynomial $(P_n\circ\phi)_n$. The other side of the coin we have described in Note \ref{NExpGrowth} is that  the function $Q_n\circ\psi$ {\bf can be used to  extrapolate}  the coefficients of $P_n$ with high accuracy. Indeed, essentially the same calculation as in Note \ref{NExpGrowth} shows that, if the $n$th Maclaurin coefficient of a function is $\mathcal{F}_{P_n,W}$ is $b_n$, this coefficient differs from the calculated $(Q_n\circ\psi)^{(n)}(0)$ by an exponentially small relative error, $\psi'(0)^{n}b_n$.

   {\em Prima facie}, discarding a specific part of the information as the best way to obtain more information may seem even more paradoxical than what we discussed in Note \ref{NExpGrowth}, where recall that we showed that it is {\it necessary} to truncate $P_n\circ\phi$ after $n$ terms, i.e., as $(P_n\circ\phi)_n$. Some further comments may help to demystify this phenomenon. If a function $F$ is analytic in $\DD$,  and if $\partial\DD$ is a natural boundary and nothing more is known about $F$ except $P_n$, then it is intuitively clear that $P_n$ is already the best approximant of $F$ (this is shown rigorously in the proof of Theorem \ref{T1}). Hence the natural place to analyze Maclaurin approximations is inside the unit disk $\DD$. Therefore, if we kept $P_n\circ\phi$ without truncation the approximation would be worse (furthermore, $P_n\circ\phi\circ\psi$ is the identity resulting evidently in no change in accuracy).

We illustrate this procedure on functions which  are analytic on $\Omega\left(\hat\CC\setminus\{(0),1,\infty\}\right)$ (see Definition \ref{DOmega()}), but it works similarly on other Riemann surfaces. In this example, the maps $\psi$ and $\phi$ are given in terms of the elliptic nome function, as in \eqref{eq:inome0} and \eqref{eq:Nome1}.

	Then we obtain exponentially accurate results for the coefficients of the extrapolation of $P_n$ to $\tilde{P}_n=\sum_{k=0}^mp_k\omega^k$, where for $k>n$ we use  $p_k=[(P_n\circ\phi)\circ\psi]^{(k)}(0)$. To illustrate how efficient this extrapolation is, consider the previously discussed  test cases, $\log(\omega^{-1}\log(1-\omega))$ and the complete elliptic integral of the first kind, $\mathbb K(\omega)$.
Then if we start with a Maclaurin polynomial of degree 8, the 9th coefficient is predicted with relative error $\epsilon_9\sim 5\cdot 10^{-9}$ and $10^{-9}$, respectively.

Even more strikingly, the subsequent 471 (!)\footnote{There is nothing intrinsically special about the order 471; it relates to the degree of a large-order Maclaurin polynomial already stored on the computer.}  coefficients are predicted with maximum relative error of $0.3\%$ and $0.13\%$, respectively. For $\mathbb K(\omega)$, this $480$-term extrapolated series is shown in Figure \ref{PlotEllipticK} as the dashed-blue curve, and is indistinguishable from the optimally extrapolated function.

	Given  instead the first 60 coefficients, the 61st coefficient is predicted with relative error $$\epsilon_{61}\sim5\cdot 10^{-73}\ \text{and } 10^{-71}\ , \ \text{respectively,}$$
	and the next 59 are predicted with maximum relative error $10^{-51}$ and $5 \cdot 10^{-49}$, respectively! 
	
		Even Pad\'e approximants can achieve impressive extrapolation: the same 60 additional coefficients are predicted with relative errors $<10^{-13}$.  For Pad\'e, the procedure  consists of calculating $P_{120}\left(\Pi_{30,30}(F)\right)$ where $\Pi_{30,30}(F)$ is the diagonal [30,30] Pad\'e approximant of one of the functions above. The reason of the improvement is similar to that described above, but since convergence of Pad\'e occurs in the sense of capacity theory only, the corresponding precision statements are only qualitative.

	The accuracy of coefficient extrapolation depends of course on the complexity of the Riemann surface $\Omega$. On one of the most intricate surfaces in our examples, the Riemann surface $\Omega_{\ZZ}$ relevant to $P_{\rm I}$ discussed in Section \ref{sec:P1-unif}, our optimal procedure predicts the $61st$ term instead with $\sim 4\cdot 10^{-25}$ relative error  (Pad\'e can also achieve $\sim 10^{-18}$). In the case of solutions of ODEs, such as $P_{\rm I}$, significant further improvement can be obtained by singularity elimination, \S\ref{sec:elim}.

	This accuracy can be further increased by bootstrapping the information:. If $\psi(\omega)=\sum_{k\in\NN}\kappa_n \omega^n$ and $F\circ\phi=\sum_{k\ge 0}f_n z^n$ and $\{f_k\}_{0\le k\le n}$ are given, then the error in the next coefficients can be calculated in closed form,
	$$ \kappa_1^n f_n ;\ \ \kappa_1^{n+1}f_{n+1}+ n \kappa_1^{n-1} f_n; \ \ \kappa_1^{n+2} f_{n+2}+(n+1)\kappa_2\kappa_1^{n+1}f_{n+1}+\left(n \kappa_3+\frac{n(n-1)}{2}\kappa_2^2\right)\kappa_1^{n-1}f_n,...$$
	Then, relying on the exponential accuracy of previously calculated coefficients, these can be used to compensate some of the errors in subsequent ones. We will however not pursue this method further here.

An interesting potential application of these series extrapolation methods would be to find accurate critical exponents by matching the asymptotic behavior of coefficients.

\subsection{Resurgent Functions}
\label{sec:resurgent}

 Resurgent functions are ubiquitous as solutions of equations in analysis, fundamentally due to the
closure of the family of resurgent functions under {\em virtually all} operations used in solving analytic equations (see \cite{costin-book} and references therein). There is also a growing body of evidence, both numerical and analytical, for resurgent functions in physical systems, for example in quantum mechanics, quantum field theory, statistical field theory and string theory
  \cite{voros,delabaere,Delabaere,Berry,ZinnJustin:2004ib,marino-matrix,Gaiotto:2009hg,Garoufalidis:2010ya,Aniceto:2011nu,Dunne:2012ae,Dunne:2016nmc,Dorigoni:2015dha,Aniceto:2015rua,Gukov:2016njj,kontsevich}.

Many classes of problems are known to have resurgent solutions. Take, as a first example, solutions of generic (see \cite{Duke})  meromorphic ODEs that admit asymptotic series at infinity.  After normalization of the variables, the system can be brought to the standard form, $y'(x)=\Lambda y +x^{-1} By+g(y; x)$, where $\Lambda$ and $B$ are diagonal matrices with constant coefficients, and $g$ contains the higher order terms in the linearization, as well as the nonlinear terms. 
The data $\Lambda$ and $B$ from the linearized problem determines the Riemann surface of $F$ as well as the nature of the Borel plane singularities of $y$ (\cite{Ecalle, Sauzin, Duke}).
 These Borel plane singularities are all of the form indicated in (\ref{eq:res-sing}). Difference and q-difference equations of roughly the same form  \cite{Braaksma,Ramis}, and also solutions of special classes of PDEs \cite{PDE1,PDE2,PDE3,PDE4,PDE5}, are also known to have resurgent solutions, and similar comments and methods apply.

 In this paper we only partially exploit this rich algebraic structure provided by \'Ecalle's bridge equations.  Further implications of this algebraic structure will be described in future work.

\subsection{Probing a Riemann surface empirically}\label{Sempir}
In the cases when the Riemann surfaces $\Omega$ need to be reconstructed empirically by the methods described in the following sections, uniformization provides a very sensitive tool for verifying this $\Omega$. Indeed, if $\Omega$ is incorrect, then $\psi$ places singularities inside $\DD$, and these are clearly  visible in an empirical $n$th root test.

\section{Singularity elimination}
\label{sec:elim}

In this Section we introduce a new {\it singularity elimination} procedure, which is a constructive two-step process: (i) an invertible linear operator (a convolution) first transforms a given singularity into a square root singularity while keeping its position fixed;  (ii) composition with a suitable conformal map then eliminates the singularity,  again, preserving the location.
The principal motivation underlying this analysis is to develop new methods to give precise determinations of the location and 
nature of a chosen singularity, and information about the local behavior of the function near this chosen singularity. This procedure provides a new method to probe, with extremely high precision, the singularities of functions which are (or are suspected to be) resurgent. These  have {\em elementary singularities} of the form in \eqref{eq:res-sing}. The singularity elimination operator can, in principle, be used to eliminate more and more singularities,  but in applications it is generally more efficient to  eliminate singularities of interest  one at a time,
leaving the others essentially unchanged.

For such functions, we show that there exist  invertible linear operators that regularize the singularities, in the sense of transforming them into points of analyticity. 
In general, these  operations cannot be reduced to conformal maps.  For example, a function with a singularity of the type $\log (1-\omega )A(\omega)+B(\omega)$, with $A, B$ holomorphic at $\omega=1$, cannot be composed with a holomorphic map $\phi$ with $\phi(0)=0,\phi(1)=1$ such that the composition is analytic at $\omega=1$. The proof is straightforward\footnote{Indeed, taking $A=1, B=0$, and then $A=1$ and $B(\omega)=\omega$, we see that $\phi$ must be analytic at $1$, hence $\phi(1+s)=1+o(s)$. But then $\log(1-\phi)$ is unbounded at $1$.}. 
Uniformizing the surface of the log at $\omega=1$ results in moving the singularity from $1$ to infinity.
However, we show that instead it is possible to construct  linear operators which simply remove the log singularity while mapping the interval $J=[0,1]$ (where they are nondecreasing)  onto itself, hence mapping $0$ to $0$ and $1$ to $1$, and a neighborhood of $J$ conformally onto a neighborhood of $J$. Hence the operators remove the singularity without moving it. 
That these methods are distinct from complete uniformization also follows from the fact that any function of the form
\begin{equation}
  \label{eq:sumbr}
  \sum_{k=1}^N c_k (\omega-\omega_k)^\alpha
\end{equation}
(with a common exponent $\alpha$)
can be reduced to a rational one by the simple procedures described below, but its Riemann surface is not uniformizable by any simple map for general $\omega_k$. 

\subsection{Properties of convolution and the singularity elimination procedure}
\label{sec:elim-th}

In this section we analyze the general structure of singularities of Laplace convolution, see \eqref{eq:defconvl}, and describe the process of elimination of elementary singularities. We begin with some definitions and a description of the procedure. We also analyze the singularities of convolutions of general analytic functions $F, G$ in the neighborhood of the singularities of $F, G$.
\begin{Definition}\label{Delim}{\rm 
 
Laplace convolution of $F$ and $G$ is defined as

    \begin{equation}
      \label{eq:defconvl}
      (F*G)(\omega)=\int_0^\omega F(s)G(\omega-s)ds
    \end{equation}
    }\end{Definition}
    
\begin{Note}\label{N:sing1}
{\rm  
 \z    {\bf \refstepcounter{minisection1}  \label{i:choicebeta} \arabic{minisection1}. } 
       For $\Re\beta>-1$,  an important role is played by the linear operator $L_\beta$ of convolution with $\omega^\beta$, followed by multiplication by $\omega^{-\beta}$:
       \begin{equation}
         \label{eq:defLbeta}
(L_{\beta} F)(\omega) =\omega^{-\beta}\int_0^\omega F(s) (\omega-s) ^{\beta} ds         
       \end{equation}
As shown below in Lemma \ref{L:Rsurface-convo} and  Lemma \ref{L:manip}, if $F$ is analytic in $\DD$ with an elementary singularity of type \eqref{eq:res-sing} at $\omega_0$, then  $L_{\beta}F$ is also analytic in $\DD$ with an elementary singularity of type \eqref{eq:res-sing}, where $\alpha$ is replaced by $\alpha+\beta+1$. In other words, the convolution operator $L_\beta$ in \eqref{eq:defLbeta} allows us to modify the nature of the chosen singularity.  Note that $L_\beta$ is both linear and invertible. 

\z {\bf  \refstepcounter{minisection1}  \label{i:squareroot} \arabic{minisection1}. } 
 Let us normalize so that the chosen singularity is at $\omega_0=1$. For singularity elimination, it is convenient to  choose $\beta$ so that  $\alpha+\beta+1=k/2$, for some nonnegative odd integer $k$, so that the singularity of  $L_{\beta}F$ becomes a square-root branch point
 \begin{equation}
         \label{eq:newtype}
 (L_{\beta} F)(\omega)  =  \,  (1-\omega)^{k/2}A_1(\omega)+B_1(\omega)
       \end{equation}
 in a neighborhood of $1$ where the functions  $A_1, B_1$ are analytic at $\omega=1$.

\z {\bf \refstepcounter{minisection1} \label{mini12} \arabic{minisection1}. }   Let  $\omega=z\mapsto \phi(z)$ denote a conformal map of the unit disk $\DD$ to some Riemann surface $\Omega$, mapping $0$ to $0$, and $1$ to $1$, and such that $(\phi-1)$ has a double zero at $z=1$. Simple examples of such maps are 
      \begin{equation}
       \label{eq:r1r2}
\phi_0(z)=2z-z^2;\ \     \phi_1(z)=   \frac{4z}{(1+z)^2};\ \ \text{ and } \phi_2(z)= \frac{2z}{1+z^2} \quad .
\end{equation}

 \z {\bf \refstepcounter{minisection1} \label{mini13} \arabic{minisection1}. }   Replacing in \eqref{eq:newtype} $\omega$ by any $\phi(z)$ for any $\phi$ in \eqref{eq:r1r2} we see that $(L_\beta F)\circ\phi (z)$  is now analytic at both $z=0$ and $z=1$.

 \z {\bf \refstepcounter{minisection1} \label{mini131} \arabic{minisection1}.} The inverse functions of the maps above are elementary, and inverting  $(L_\beta F)\circ\phi$
  results in a  representation of $F$ as a series of special functions. But even if, for a more complicated $\phi$, inverting convolution can only be done by some (convergent in the limit) numerical scheme, we reiterate that the purpose here is to most accurately recover an unknown function $F$ from truncated Maclaurin series, and not necessarily to provide economical approximations of a known function.
     
 \z  {\bf \refstepcounter{minisection1} \label{mini14} \arabic{minisection1}. }   Lemma \ref{Lsing} below generalizes this idea to much more general singularities than the elementary singularities in \eqref{eq:res-sing}, and can be used for their elimination. However, in this paper we will not attempt to classify in general the singularities that can be eliminated.
    
 }   \end{Note}
 
\subsection{Singularity Transformation}
 \label{sec:sing-transf}

\begin{Lemma}[Preservation of Riemann surfaces by $L_\beta$] \label{L:Rsurface-convo} Assume $\Omega=\Omega\left(\hat\CC\setminus\{(0), S\}\right)$ (recall Definition \ref{DOmega()}). We generalize $L_\beta$ when $\Re \beta>-1$, by interpreting the convolution in \eqref{eq:defLbeta} as an integral along a smooth curve $\gamma:[0,1]\to \Omega$ with $\gamma(0)=0$ and $\gamma(1)=\omega$.

Then, if $F$ is analytic on $\Omega$,  $L_\beta F$ is also analytic on $\Omega$.  
\end{Lemma}
\begin{proof}
  Analyticity in $\DD$ is easily shown by replacing $F$ by its Maclaurin series, and then using dominated convergence to integrate term by term. Then we calculate
  \begin{equation}
    \label{eq:convopow}
 L_{\beta}\left(\sum_{k=0}^{\infty} a_k \omega^k\right)= \sum_{k=0}^{\infty} \frac{\Gamma (\beta+1) \Gamma (k+1) }{\Gamma (\beta+k+2)}\, a_k \, \omega^{k+1} \quad ; \quad |\omega|<1
  \end{equation}
  For large $k$ we have
  $$ \frac{\Gamma (k+1)}{\Gamma(\beta+k+2)}=O(k^{-1-\beta}) $$
 implying that  analyticity in $\DD$ is preserved.

  Next, examining \eqref{eq:defLbeta} we interpret $F(s)ds$ as a bounded complex measure along any path from $0$ to $\omega$ and note that the integrand, and hence the integral over $\gamma$ are manifestly analytic at all points in $\gamma([0,1))$. Let $\epsilon>0$ be such that $\DD_{2\epsilon}(\omega)\subset \Omega$, let $\omega_1\in\DD_{\epsilon}(\omega)\in \gamma([0,1))$ and write the integral along $\gamma$ as $\int_0^\omega =\int_0^{\omega_1}+\int_{\omega_1}^{\omega}$. The first integral is analytic by the argument above. The second integral can be replaced by a straight line integral from $\omega_1$ to $\omega$, where we change variable $s=\omega-t$ to get
  $$\int_0^{\omega-\omega_1}F(\omega-t)t^\beta dt$$
  which is analytic since $F$ is analytic in $\DD_{2\epsilon}(\omega)$
\end{proof}
\begin{Note}
  {\rm Convolving more general functions with singularities {\bf alters} the Riemann surface, in general. For instance $(1-\omega)^{1/2}*(1-\omega)^{1/2}$, is an elementary function with a square root type singularity at $1$ and a log-type singularity at $2$. We refer to \cite{Sauzin} for the theory of the location of singularities generated via convolution.}
\end{Note}
For analyzing the {\bf type}  of singularities of $F*G$ for more general $F,G$ we restrict our attention to star-shaped domains $\mathcal{N}\supset \DD$ (meaning that for each point in $\mathcal{N}$,  the line segment connecting it to $0$ is also in $\mathcal{N}$)
 \begin{Lemma}[Analyticity of Convolution]
   \label{Lconvolution}
    Let $\mathcal{N}\supset \DD$ be a star-shaped domain in $\CC$. If  $F$ and $G$ are analytic in $\mathcal{N}$, then so is $F*G$.
\end{Lemma}
\begin{proof}
  This is clear if we interpret $F(s)ds$ in \eqref{eq:defconvl} as a finite complex measure on compact sets of $\mathcal{N}$.

  \end{proof}
  The next Lemma describes how singularities at $0$ and $\omega_0$ interact. For simplicity of notation, we {\bf normalize} $\omega$ so that $\omega_0=1$. 
    
  \begin{Lemma}[{Calculation of singularities of convolution}]\label{Lsing}
    \begin{enumerate}
    \item  Assume $F$ is analytic in a neighborhood of $(0,1+\epsilon]$, possibly singular at zero but $F \in L^1((0,1+\epsilon))$, $G$  is analytic in a neighborhood of $[0,1+\epsilon]\setminus[1,1+\epsilon]$ with continuous lateral limits (possibly different) on the cut. 
   Assume further that there exist two functions, $S$ analytic in $\DD_{\epsilon}(1)\setminus [1,1+\epsilon]$, and $B_1$ is analytic in the disk $\DD_{\epsilon}(1)$ such that
     \begin{equation}
     \label{eq:singexp}
 (F*G)_1(\omega):=    \int_1^\omega F(\omega-s)G(s)ds=S(\omega)+B_1(\omega)
   \end{equation}
   Then,
     \begin{equation}
     \label{eq:convosing}
\int_0^\omega F(\omega-s)G(s)ds=S(\omega)+B(\omega)
\end{equation}
where $B(\omega)$ is analytic in $\DD_{\epsilon}(1)$.

The result can be adapted to the case where instead $G=G_1^{(k)}$ for some $k\in \NN$, and $G_1$ satisfies the assumptions of the lemma.
  
\item The result extends to  Riemann surfaces $\Omega$ as in Lemma \ref{L:Rsurface-convo} if $F(\omega)=\omega^{\beta}$, $\Re \beta >-1$. More precisely, recalling that we normalized $\omega_0$ so that its projection on $\CC$ is 1, we choose a line segment in $\Omega$ emanating from $\omega_0$ whose projection in $\CC$ is $[1,1+\epsilon)$ and replace the branch jump across $[1,1+\epsilon)$ by the branch jump across this segment.
    \end{enumerate}
  \end{Lemma}
  \begin{proof}[Proof of Lemma \ref{Lsing}]
1.   We note that the jumps across the cut $[1,1+\epsilon)$ of  $\int_0^\omega F(\omega-s)G(s)ds$ and of $S$ must coincide, and it evidently also coincides with the jump across the cut\footnote{As usual, by ``jump across the cut'' we mean the upper limit minus the lower limit along the cut.} $[1,1+\epsilon)$ of  $\int_1^\omega F(\omega-s)G(s)ds$ (since the integrand is analytic in $\DD$. Using Lemma \ref{Lconvolution} and the assumptions of Lemma \ref{Lsing}, we see that $B(\omega):=\int_0^\omega F(\omega-s)G(s)ds -S(\omega)$ is analytic in $\DD_{\epsilon}(1)\setminus [1,1+\epsilon]$ and continuous in
  $\DD_{\epsilon}(1)$, and Morera's theorem implies that $B$ is analytic in  $\DD_{\epsilon}(1)$.

  For a general $k$ we write $\int_0^\omega=\int_0^{1/2}+\int_{1/2}^\omega$. In the second integral we integrate by parts $k$ times to eliminate the derivatives of $G_1$ and apply the first part.

     2. Follows in the same way, noting that only a neighborhood of $\omega_0$ is involved in both the statement and in the proof of 1. 
\end{proof}
\begin{Note}
{\rm 

  Lemma \ref{Lsing} is a ``localization'' lemma, whose point is that $S(\omega)$ can be calculated from local expansions of $F$ and $G$ at $0$ and $1$ respectively, whenever such expansions exist\footnotemark, since $\int_1^\omega F(\omega-s)G(s)ds$ only depends on these local expansions.
}\end{Note}
\begin{Note}{\rm
    \begin{enumerate}
\item
    
Clearly, for the elementary singularities of type \eqref{eq:res-sing}, the local expansion assumed in \eqref{eq:singexp} always exists.
(Note that other singularities of $F$ and $G$, including subtler singularities on the second Riemann sheet, may mean that the local expansion of $F$ at zero and that of $G$ at $1$ have radius of convergence $<1$.)

   \item
The convolution operator $L_{\beta}$ is important in applications, as it gives a simple way to transform the nature of the singularity. Lemma \ref{L:manip} below shows that the effect of this operator can be implemented directly on the original expansion coefficients.

\item
 The transformation of singularities can also be understood in terms of fractional derivatives, as is clear from the representation in (\ref{eq:defLbeta}). See \cite{gokce} for an application to Borel transforms. 
   \end{enumerate}
   
}\end{Note}
\footnotetext{Very generally, even in the absence of local expansions, Plemelj's formulas \cite{Ablowitz}  give a local representation $S(\omega)$ in the form $ \frac{1}{2\pi i}\int_0^{1+\epsilon}(\omega-\tau)^{-1}\int_1^\lambda F(\lambda-s)\Delta G(s)dsd\tau$, where  $\Delta G$ is the jump across the cut of $G$. }

\begin{Lemma}\label{L:manip}
    Assume $F$ is analytic in $\DD$, and at $\omega=1$ it has an elementary singularity of type \eqref{eq:res-sing}.  Then, for $\Re\beta>-1$ the convolution $L_{\beta} F$ in (\ref{eq:defLbeta}) has an elementary singularity of type \eqref{eq:res-sing} with $\alpha$ replaced by $\alpha+\beta+1$.

\end{Lemma}
\begin{proof}
 We rely on Lemma \ref{Lsing}. Take first $\alpha\notin\ZZ$.  In view of the second part, we may assume $\Re\alpha>-1$. In our case, for small enough $t$,
   \begin{equation}
           \label{eq:explf}
    S(\omega)= -2i\sin(\pi \alpha)(1-\omega)^{\beta+\alpha+1}\int_0^1(1-t)^{\beta}t^{\alpha}A(1+t(1-\omega))+B_1(\omega)
  \end{equation}
  as seen by the change of variable $s=t \omega$.
 Writing near $1$, the local expansion $A(\omega)=\sum_{k=0}^\infty a_k (1-\omega)^k $ and inserting in \eqref{eq:explf}, we see that
\begin{equation}
  \label{eq:hAtOne}
S(\omega)= (1-\omega)^{\beta+\alpha+1} \sum_{k=0}^\infty a_k\frac{\sin(\alpha\pi) \Gamma(\beta+1)\Gamma(\alpha+k+1)}{\Gamma(\beta+\alpha+k+2)\, \sin((\beta+\alpha)\pi)}(1-\omega)^k
\end{equation}
proving the assertion in this case. (Note the obvious analogy to singularities of hypergeometric functions.)

\item

For a logarithmic singularity, $F(\omega)=\ln(1-\omega) A(\omega)+B(\omega)$, we have 
\begin{equation}
           \label{eq:explflog}
( L_\beta F)(\omega)=\omega^{-\beta}\int_{0}^{\zeta}\! \left( \zeta-u \right) ^{\beta}\ln  \left( -u \right) A \left( 
1+u \right) \,{\rm d}u\quad, \quad \zeta =\omega-1
\end{equation}
and the jump across the cut of $L_\beta F$ is 
\begin{equation}
  \label{eq:logjump}
S(\omega)=2\pi i \omega^{-\beta}\int_0^\zeta\left( \zeta-u \right) ^{\beta} A \left( 
1+u \right) \,{\rm d}u
\end{equation}
Writing $A(\omega)=\sum_{k=0}^\infty a_k (1-\omega)^k$ near $1$, one can verify that, in order to achieve the same branch jump with an $S$ having a cut $[1,\epsilon)$ we define  $S(\omega) =2\pi i \omega^{-\beta}(\omega-1)^{\beta+1}(1-e^{2\pi i \beta})^{-1}\sum_{k=0}^\infty b_k (\omega-1)^k$ where $(\omega-1)^{\beta+1}$ is defined to be positive on the upper part of the cut $[1,\infty)$ and
\begin{equation}
  \label{eq:bk}
b_{k}= \Gamma \left(\beta  +1 \right)  (-1)^k\, {\frac {\Gamma \left(k +1\right)}{\Gamma \left( \beta +k+2 \right) }}\, a_k
\end{equation}
proving the statement when $\alpha=0$. Using the second part of Lemma \ref{Lsing}, the general case follows from it by integration by parts in the branch jump formula. An explicit example is shown in \S\ref{sec:sing-elim-ex}.
\end{proof}

\subsection{Singularity Elimination Theorem}
\label{sec:sing-elim-theorem}
 \begin{Theorem}[Singularity Elimination]
  \label{thelim}
Assume $F$ is analytic on $\Omega=\Omega\left(\hat\CC\setminus \{(0) S\}\right)$ (recall Definition \ref{DOmega()}),  that $\omega\in\partial \Omega$ and that $F$ has a singularity of type \eqref{eq:res-sing} at $\omega$. Without loss of generality, we can assume that the projection of $\omega$ on $\hat\CC$ is $1$. Then the singularity can be eliminated by a combination of an appropriate $L_\beta$ and a composition with a rational map such as those in \eqref{eq:r1r2}. 

\end{Theorem}

\begin{proof}[Proof of Theorem \ref{thelim}] 
The proof follows from Lemmas \ref{Lconvolution},  \ref{Lsing}, and items \ref{i:choicebeta}, \ref{mini12} and \ref{mini13} of Note \ref{N:sing1} at the beginning of \S\ref{sec:elim-th}.
 \end{proof}
 
\begin{Note}[Comments on Theorem \ref{thelim}]
\label{Nelim}{\rm 

\begin{enumerate}

\item
   Theorem \ref{thelim} yields a  practical method to apply simple convolution and conformal maps to make the local behavior near a singularity purely analytic. Since analyticity and singularity are highly sensitive to being distinguished numerically, this therefore provides a numerical mechanism to refine both the location of the singularity and also to refine the convolution parameter $\beta$, in order to determine the power exponent $\alpha$ which characterizes the nature of the original singularity. This is particularly useful when  empirical analysis is the only option available. See examples in  \S \ref{S:empirical} and \S \ref{S:Applic}.
     
   \item   After eliminating a singularity (say at $\omega$) of $F$, if we uniformize the Riemann surface of the new function $F_1$, then $z=\psi(\omega)\in \DD$,  $F\circ\phi$ is analytic at $z$ and can be calculated convergently and with rigorous bounds. This means that the complete information about the singularity of $F$ (such as the functions $A,B$ if the singularity is of type \eqref{eq:res-sing}) follows.  Uniformization of the new surface may be impractical, in which case an appropriate Conformal-Taylor expansion (see \S\ref{CT}) would provide, with sub-optimal rate of convergence, the same information (or, non-rigorously, using Pad\'e approximants,  see \S\ref{S51}). 
   
   \item
In practical computations we noticed that the precision of the local information obtained from singularity elimination is usually significantly better than what is obtained numerically from an explicit uniformization map. See for example the Painlev\'e I computation described in  Note \ref{N:02}.
  
  \item
  There are important special cases in which all singularities are eliminated,  for example arrays of pure singularities of the type  $\sum_{k=0}^n (\omega-\omega_k)^{\alpha}$, which can be transformed by an appropriate $L_\beta$ into a sum of logs, which becomes rational after differentiation. 
   
   \item   Still for empirical analysis, for all three maps in (\ref{eq:r1r2}) analytic continuation past $1$ of $F\circ\phi$ leads to the second Riemann sheet of $F$. For example, the map $\phi_0$ takes the origin of the second Riemann sheet of $F$ to  $\omega=2$ on the first Riemann sheet of $\phi_0^*F$. Therefore, singularity elimination also provides access to higher Riemann sheets. This will be an important element of the tools developed in \S \ref{S:empirical} to refine approximate data about the Riemann surface.
   
   \end{enumerate}
  
 } \end{Note}

\begin{Note}[Important special cases  of singularity transformation and elimination]{\rm

 For example, we can manipulate and eliminate the log singularity at $\omega=1$ of the elliptic integral function $\mathbb K(\omega)$. See \S\ref{sec:sing-elim-ex}.

}\end{Note}

\subsection{The counterpart on series of the singularity elimination operator}
\label{sec:conv-series}

 Both steps of the analytic operations described above for singularity elimination have a simple and explicit counterpart as operations at the level of the series coefficients, taking Maclaurin polynomials to Maclaurin polynomials. This is important in applications, since series compositions with many coefficients is computer-algebra time-expensive. Recall first the explicit expression \eqref{eq:convopow} for the action of the convolution operator $L_\beta$ on series. Here we consider the second step, that of singularity elimination, and derive explicit formulas for the composition with the elimination maps $\phi_0$, $\phi_1$ and $\phi_2$ listed in equation (\ref{eq:r1r2}) of item \ref{mini12} of Note \ref{N:sing1}.
\begin{Lemma}\label{L:18}
  On the level of series, the operator $\phi_0^*$ of composition with the conformal map $\phi_0$ is given by

  \begin{enumerate}
  \item  $\phi_0^*(\sum_{k=1}^n a_k \omega^k)= \sum_{k=1}^n b_k \omega^k$ where
     \begin{equation}
      \label{eq:bkforR1a}
   b_0=a_0;\ \    b_k =\sum_{l=0}^{\lfloor (k+1)/2 \rfloor} U_{k,l}a_{l}
 \end{equation}
 where $U_{k,l}$ is the coefficient of $x^l$ in the $k$th Chebyshev polynomial of the second kind $U_k(x)$, explicitly,
  \begin{equation}
   \label{eq:coef-gammaform}
   U_{k,l}= \begin{cases} \displaystyle{
 (-1)^l 2^{2 l+1} \binom{\frac{k+1}{2}+l}{\frac{k-1}{2}-l},\ \ \text{$k$ odd}}
 \\ \displaystyle{-4)^l \binom{\frac{k}{2}+l}{\frac{k}{2}-l}\text{,\quad  $k$ even}}
 \end{cases}
\end{equation}

\item $\phi_1^*$ acts by $\phi_1^*(\sum_{k=1}^n a_k \omega^k)= \sum_{k=1}^n b_k \omega^k$ where
     \begin{equation}
      \label{eq:bkforR1b}
   b_0=a_0;\ \    b_k =4\sum_{l=1}^k \tilde{U}_{k,l}a_l
    \end{equation}
    where $\tilde{U}_{k,l}$ is the coefficient of $x^l$ in the $k$th Chebyshev polynomial $U_k(2x-1)$.
   
  \item $\phi_2^*$ acts by  $\phi_2^*(\sum_{k=1}^n a_k \omega^k)= \sum_{k=1}^n b_k \omega^k$, where the new series coefficients $b_k$ are related to the original series coefficients $a_k$ via
     \begin{equation}
      \label{eq:bkforR1c}
   b_0=a_0;\ \    b_k=2\sum_{l=0}^k U_{k,l}a_{2l+1};\quad (k\ \text{odd),  and }   b_k =2\sum_{l=1}^k U_{k,l}a_{2l};\quad (k\ \text{even})
 \end{equation}
  \end{enumerate}
 
\end{Lemma}
\begin{proof}
  We prove only part 3., since all three proofs are very similar. Define $\displaystyle{f(\omega)=\frac{\omega}{2(1-x\omega)}}$. Then,
  \begin{equation}
    \label{eq:gfun}
    (\phi_2^*f)(\omega)=\frac{\omega}{1-2 x \omega +\omega^2}
  \end{equation}
  which is, up to multiplication by $\omega$, the known generating function of $U_k(x)$:
  $$\displaystyle \sum _{k=0}^{\infty }U_k(x)\omega^{k}={\frac {1}{1-2 x \omega+\omega^{2}}}$$
  and the rest follows easily by comparing coefficients. 
\end{proof}

\begin{Note}[Remarks on numerical accuracy]{\rm 

\begin{enumerate}
\item
Since the calculation of the $b_k$ from the $a_k$ involves summations, there can be cancellations, which could potentially become significant for large $k$.  The following lemma addresses the question of the accuracy of the $b_k$, or even how many coefficients can be meaningfully retained. This is relevant also for selecting which map results in the minimal loss of accuracy. This is an interesting question for which examples will be given in an accompanying paper
 \cite{apps-paper}, and for which rigorous estimates are under investigation.

\begin{Lemma}
  For fixed large $k$, $|U_{k, l}|$ reaches its maximum value $M_k$ at $l\sim \frac{1}{2 \sqrt{2}}k+\frac{1}{8} \left(2 \sqrt{2}-5\right)$
  $$M_k\sim \frac{2^{-1/4}}{\sqrt{\pi k}}\left(1+\sqrt{2}\right)^{k+1}\qquad , \quad k\to\infty$$
\end{Lemma}
\begin{proof}
This result can be derived by asymptotically solving the equation $U_{k, l}/U_{k,{l-1}}=1$ for $l=l(k)$, and using Stirling's formula for large $k$ in $U_{k,l(k)}$.
\end{proof}
\z For example, using $\phi_0^*$, taking into account the position of the maximum, for large $k$ the coefficient $b_k$ involves cancellations of terms $a_m$, with $m\le k$ weighted by $ \left(1+\sqrt{2}\right)^{2 m}\approx 6^m$.

\item
In general, the accuracy needed can be calculated similarly, from the capacity $C$.

\end{enumerate}
}\end{Note}

\subsection{Singularity Elimination Example}
\label{sec:sing-elim-ex}

In this section we illustrate the general procedure of singularity elimination by using a suitable convolution operator $L_\beta$  in (\ref{eq:defLbeta}) to transform the logarithmic singularity of the elliptic integral function $\mathbb K(\omega)$ into a square root singularity, and then eliminating this singularity by composition with a suitable conformal map.  We choose this example because of its practical interest and also because we can compare the general expressions in \S\ref{sec:elim} with analytic results for the transformation of hypergeometric functions.
We define
\begin{eqnarray}
F(\omega)=\mathbb K(\omega)=\frac{\pi}{2} \,~_2 F_1\left(\frac{1}{2}, \frac{1}{2}, 1; \omega\right)
\label{eq:k}
\end{eqnarray} 
The expansions as $\omega\to 0^+$ and $\omega\to 1^-$ are:
\begin{eqnarray}
F(\omega)&\sim & \frac{1}{2} \sum_{k=0}^\infty \left(\frac{\Gamma\left(k+\frac{1}{2}\right)}{\Gamma(k+1)}\right)^2\, \omega^k \qquad, \quad \omega\to 0^+ 
\label{eq:F0}
\\
F(\omega) &\sim & A(\omega)\, \ln(1-\omega)+B(\omega) \qquad, \quad \omega\to 1^- 
\label{eq:F1}
\end{eqnarray}
where at the logarithmic singularity the regular functions $A(\omega)$ and $B(\omega)$ behave as
\begin{eqnarray}
A(\omega)&=& -\frac{1}{\pi} \, \mathbb K(1-\omega) =  -\frac{1}{\pi} \sum_{k=0}^\infty \left(\frac{\Gamma\left(k+\frac{1}{2}\right)}{\Gamma(k+1)}\right)^2\, (1-\omega)^k \qquad, \quad \omega\to 1^-
\label{eq:A}
\\
B(\omega)&=& -\sum_{k=0}^\infty \left(\frac{\Gamma\left(k+\frac{1}{2}\right)}{\Gamma(k+1)}\right)^2 \left(\psi\left(k+\frac{1}{2}\right)-\psi(k+1)\right) (1-\omega)^k \qquad, \quad \omega\to 1^-
\label{eq:B}
\end{eqnarray}
The convolution operator (\ref{eq:defLbeta}), with $\beta\notin \mathbb Z$,  acting on $F$ leads to
\begin{eqnarray}
(L_\beta F)(\omega) &=& \frac{\pi}{2(1+\beta)} \,\omega\, ~_2 F_1\left(\frac{1}{2}, \frac{1}{2}, 2+\beta; \omega\right)
\label{eq:LbetaF1}
\\
&=& 
 \frac{\Gamma(1+\beta)}{2} \sum_{k=0}^\infty \frac{\Gamma\left(k+\frac{1}{2}\right)^2}{\Gamma(k+1) \Gamma(2+k+\beta)}\, \omega^{k+1} \qquad, \quad \omega\to 0^+ 
 \label{eq:LbetaF2}
 \\
&\sim & \tilde{A}(\omega)(1-\omega)^{1+\beta}+\tilde{B}(\omega)\qquad, \quad \omega\to 1^-
\label{eq:LbetaF3}
\end{eqnarray}
where $\tilde{A}(\omega)$ and $\tilde{B}(\omega)$ are analytic at $\omega=1$.
Equation (\ref{eq:LbetaF2}) confirms the general expression (\ref{eq:convopow}) relating the original expansion coefficients of $F(\omega)$ with those of the convolved function $(L_\beta F)(\omega)$. To transform the original logarithmic singularity at $\omega=1$ to a square root behavior at $\omega=1$ we choose $\beta=-\frac{1}{2}$, to obtain
\begin{equation}
(L_{-\frac{1}{2}} F)(\omega) = \pi \,\omega\, ~_2 F_1\left(\frac{1}{2}, \frac{1}{2}, \frac{3}{2}; \omega\right) =\pi\, \sqrt{\omega} \, {\rm arcsin}(\sqrt{\omega}) 
\label{eq:LbetaF121}
\end{equation}
(Of course, in this special case composition with the series of $\sin^2$ results in factorial convergence of the composed series, far superior to the generic rate in the optimality theorem.)

  $L_\beta F$ is analytic at zero and singular at $\{1,\infty\}$, and at $\{0,1,\infty\}$ on higher Riemann sheets.  Its Maclaurin series is
\begin{equation}
\frac{\sqrt{\pi}}{4} \sum_{k=0}^\infty \frac{\Gamma\left(k+\frac{1}{2}\right)^2}{\Gamma(k+1) \Gamma\left(k+\frac{3}{2}\right)}\, \omega^{k+1} \qquad, \quad  \omega \in \DD
 \label{eq:LbetaF122}
\end{equation}
and its singularity structure near $\omega=1$ is
\begin{equation}
\tilde{A}(\omega)(1-\omega)^{1/2}+\tilde{B}(\omega)
\label{eq:LbetaF123}
\end{equation}
where
  \begin{eqnarray}
  \tilde{A}(\omega)&=& -\pi\,\sqrt{\omega}\, \frac{ {\rm arcsin}(\sqrt{1-\omega})}{\sqrt{1-\omega)}} = 
  -\frac{\pi}{2} \sqrt{\omega} \sum_{k=0}^\infty (-1)^k \frac{\Gamma\left(k+\frac{1}{2}\right)^2}{\Gamma(k+1) \Gamma\left(k+\frac{3}{2}\right)} (1-\omega)^k
  \label{eq:Atilde}
  \\
  \tilde{B}(\omega) &=& \frac{\pi^2}{2} \sqrt{\omega}
  \label{eq:Btilde}
  \end{eqnarray}
  We see from (\ref{eq:Atilde}) that the expansion coefficients of $\tilde{A}(\omega)$, the function in (\ref{eq:LbetaF123}) multiplying the square root behavior, match the general expression in (\ref{eq:bk}).
 In the practical situation where we only have (a finite number of) the original expansion coefficients, we simply transform the expansion coefficients according to the convolution results in \S\ref{sec:conv-series}.

 The final step of the singularity elimination is to make a composition map that transforms the square root behavior in (\ref{eq:LbetaF123}) into analytic behavior,  using the map $\phi_2$ in  (\ref{eq:r1r2}). The Riemann surface of the new function, after composition with $z\mapsto iz$ is uniformized by \eqref{eq:inome0} which brings the original singular point $-1$ inside the unit disk, where (assuming we did not know $\tilde A,\tilde B$) these could be calculated with extremely high accuracy.
 
 A similar comparison can be made for a general hypergeometric function
 \begin{eqnarray}
F(\omega)&=& ~_2F_1(a, b, c; \omega) 
\label{eq:2F11}
\\
&\sim& A(\omega) (1-\omega)^{c-a-b}+B(\omega) \qquad, \quad \omega\to 1^-
\label{eq:2F12}
\end{eqnarray}
for which the convolved function becomes a generalized hypergeometric function with a different singularity exponent at $\omega=1$:
 \begin{eqnarray}
(L_\beta F)(\omega)&=& \frac{1}{(1+\beta)\Gamma(1+\beta)} \, \omega\,  ~_3F_2({1, a, b}, {c, 2+\beta}; \omega) 
\label{eq:3F21}
\\
&\sim& \tilde{A}(\omega) (1-\omega)^{c-a-b+1+\beta}+\tilde{B}(\omega) \qquad, \quad \omega\to 1^-
\label{eq:3F22}
\end{eqnarray}
For a given original singularity exponent, $c-a-b$, a suitable choice of $\beta$ transforms the singularity into a square root singularity, which can then be eliminated by composition with one of the conformal maps in (\ref{eq:r1r2}).
\begin{Note}\label{NPartial}
	 {\rm When the nature of singularities is known a priori and a particular singularity on some Riemann sheet needs to be understood better, say by singularity elimination, partial as opposed to complete uniformization might be necessary. Indeed, upon complete uniformization, $\partial \DD$ is a natural boundary and singularity elimination there may help very little, whereas the singularities after partial uniformization are always isolated.}
\end{Note}

\section{New Approximate Methods for Empirically Probing  the Riemann Surface}
\label{S:empirical}

In this Section we address the question of how to extrapolate  and analytically continue the function $F$ when the only input is a finite number $n$ of terms of its Maclaurin series about some point, so that  the underlying Riemann surface $\Omega$ has to be determined also.
Evidently nothing rigorous can be said if $n$ is fixed, and we focus on methods that are efficient and convergent as $n\to \infty$. We present methods to determine approximate information about the singularity structure of $F$, and methods to refine and corroborate this approximate information.
We also adapt known results to provide precise rates of the convergence of these methods.

\subsection{Overview of Mathematical Results on Pad\'e Approximants}
 \label{S51}
Diagonal (and near-diagonal) Pad\'e approximation is one of the most frequently used methods for empirical reconstruction \cite{baker,bender}
\begin{Definition}
The $[m/n]$ Pad\'e approximant of $F$  at $\omega=0$ is the unique rational function $ A_m/B_n$, with $A_m$ a polynomial
of degree at most $m$, and $B_n$ a polynomial of degree at most $n$, for which we have
\begin{equation}
  \label{eq:defPade}
  F(\omega)-\frac{A_m(\omega)}{B_n(\omega)}=\mathcal{O}\left( \omega^{m+n+1}\right),\ \ \ \omega\to 0
\end{equation}
If we normalize $B_n(0)=1$, then $A_m$ and $B_n$ are also unique. Since it can be calculated directly from the Maclaurin series of $F$, we also say that $[m/n]$ is the Pad\'e approximant of $F$.

A sequence of Pad\'e approximants $\{[n/n]\}_{n\in\NN}$ is called diagonal, and $\{[m_j/ n_j]\}_{j\in\NN}$ is near-diagonal   if $n_j\to \infty$ and $m_j/n_j\to 1$ as $j\to\infty$. 
 \end{Definition}
 In spite of their simplicity (they are rational functions with the same Maclaurin series as $F$, inasmuch as their degree permits) Pad\'e approximants are, in most applications, uncannily accurate and able to detect poles and branch points in the whole complex domain (in principle).

However, except for special types of functions such as Riesz-Markov ones (see \cite{Wall1,Damanik-Simon} and references therein, and Note \ref{N:29}), they do not generally converge pointwise, but only in a weaker sense, in the sense of capacity theory. For this reason Pad\'e approximants can only be used as an exploratory tool.  Nevertheless, we explain below how these exploratory findings can be backed up rigorously, in the limit $n\to \infty$.

We briefly describe some important results (both negative and positive) concerning the convergence of Pad\'e approximants,
which seem to be little known to the applied community, outside the specialized literature. 

An intrinsic limitation is immediately clear: as any sequence of rational approximations, they can only converge in some domain of single-valuedness of their associated function.

 In fact, even for single-valued functions, uniform convergence of {\em some} diagonal Pad\'e subsequence to general meromorphic functions, the Baker-Gammel-Wills conjecture \cite{b-g-w},  was settled in the negative in a remarkable paper of Lubinsky in 2003 \cite{Lubinsky}. The phenomenon that prevents pointwise convergence are the so-called spurious  poles, or Froissart doublets, appearing at points unrelated to the properties of the associated function. In practice however, most often spurious poles appear infrequently and their exploratory value is largely unaffected.

Diagonal (and near-diagonal) Pad\'e approximations do converge in a weaker sense, namely in capacity, and in this sense they ``choose'' a maximal domain of single-valuedness where they converge, maximizing also the rate of convergence near $\omega=0$. This choice however also comes with a drawback:  points of interest of $F$ may be hidden in their boundary of convergence.  This is actually a common occurrence in applications.
We also propose new practical methods to overcome some of these limitations of Pad\'e approximants, see \S\ref{CT} and \S\ref{sec:hidden}.

\subsubsection{Convergence of Pad\'e approximants}
 \label{S52}
  Convergence of near-diagonal Pad\'e approximants to functions with branch points is  a very interesting and difficult question, only elucidated in 1997 in the fundamental paper of Stahl \cite{Stahl}. It is interesting to note that convergence in capacity is established at this time only for functions analytic on  $\Omega\left(\CC\setminus E\right)$, for sets $E$ of zero logarithmic capacity or in domains in $\CC$ bounded by piecewise analytic arcs under a stringent symmetry condition \cite{Stahl}. We focus on the first type of functions,  which are the ones of interest here.

The general theory of Pad\'e approximants summarized below is  based on \cite{Stahl}. For further developments and refinements, see \cite{Yattselev,Fink}.

The theory is best described by doing an inversion and {\bf placing the point of expansion at infinity} rather than at $\omega=0$.  
It is shown in \cite{Stahl} that there exists a domain $\mathcal{D}\subset \hat\CC$, unique up to a capacity zero set,  whose boundary  has minimal logarithmic capacity,  which contains $\infty$ and where $F$ is analytic and single valued. This $\mathcal{D}$ is the domain where near-diagonal Pad\'e approximants converge in capacity to $F$.
The rate of convergence is controlled by the Green's function $g_{\mathcal{D}}$ (see, e.g., \cite{szego,Ransford,Saff}) relative to infinity as follows. Define
$G_{\mathcal{D}}=e^{-g_{\mathcal{D}}}$. We have $G_{\mathcal D}\in [0,1)$ and $G_{\mathcal D}>0$ on $ \mathcal{D}\setminus\{\infty\}$ (in the case of interest, where cap($\mathcal{D})>0$). Then, summarizing from Theorem 1 by Stahl \cite{Stahl}, 

\begin{enumerate}
\item For any $\epsilon>0$ and any compact set $V\subset \mathcal{D}\setminus\{\infty\}$ we have
 \begin{equation}
  \label{eq:limcap1}
  \lim_{j\to\infty}\text{cap}\{\omega\in V|(F-[m_j/n_j])(\omega)>(G_{\mathcal{D}}(\omega)+\epsilon)^{m_j+n_j}\}=0
\end{equation}
\item If $F$ has branch points, which occurs iff $G_{\mathcal{D}}\ne 0$, then for any compact set $V\subset \mathcal{D}\setminus\{\infty\}$ and any $0<\epsilon\le \inf_{\omega\in V}G_{\mathcal{D}}(\omega)$ we have
 \begin{equation}
  \label{eq:limcap2}
  \lim_{j\to\infty}\text{cap}\{\omega\in V|(F-[m_j/n_j])(\omega)<(G_{\mathcal{D}}(\omega)-\epsilon)^{m_j+n_j}\}=0
\end{equation}
  
\end{enumerate}

\begin{Note}\label{N:29}{\rm 
  \begin{enumerate}

  \item In the rather generic case when $\mathcal{D}$ is simply connected, then $G_{\mathcal{D}}=|\psi_{\infty}|$, where $\psi_\infty$ is a conformal map from $\mathcal{D}$ to  $\DD$,   with $\psi_\infty(\infty)=0$. Comparing with Theorem \ref{T1}, we note that {\bf if} the maximal domain of analyticity of $F$ happens to be this $\mathcal{D}$,  the {\em geometric part of the}  rate of convergence  in capacity of Pad\'e would be optimal.

\item When $\mathcal{D}$ is simply connected, in view of 1. above and \eqref{eq:limcap1}, we see that Pad\'e effectively ``creates its own conformal map'' of a single-valuedness domain for $F$,  denoted  by $\mathcal{D}$, that can be recovered, in the limit $n\to \infty$, from the harmonic function $|G_{\mathcal{D}}|$, obtained  by taking the $n$-th root of the convergence rate \eqref{eq:limcap1}.
  
  \item  In very special cases, such as Riesz-Markov functions, under some further restrictions, the convergence of Pad\'e approximants is uniform on compact sets (cf. \cite{Damanik-Simon} and references therein). A Riesz-Markov is a function that can be written in the form
 $$F(\omega)=\int_a^b\frac{d\mu(x)}{x-\omega}$$
 where $\mu$ is a positive measure. Riesz-Markov functions occur frequently in certain applications, but general functions cannot be brought to this form. For example, a common situation in applications, discussed in more detail in \S\ref{sec:2cut} below, is the situation of two complex conjugate singularities in $\CC$. This is not a Riesz-Markov function, and Pad\'e produces curved arcs of poles (see Figure \ref{fig:two-cc-cut}) which do not relate to the properties of the function. 
 
\item As mentioned, for more general functions, Pad\'e approximants may place spurious poles (``Froissart doublets'') on sets of zero capacity, ``random'' pairs of a pole and a nearby zero,  unrelated to the function they approximate.

\item The numerators and denominators of Pad\'e approximants are orthogonal polynomials, in a generalized sense, along arcs in the complex domain, but therefore without a bona-fide Hilbert space structure. According to \cite{Stahl}, this is the ultimate source of capacity-only convergence, and of the appearance of Froissart doublets.

\item If $F$  has only isolated singularities on $\Omega\left(\hat\CC\setminus S\right)$ where $S$ is finite, $\partial \mathcal{D}$ is a set of piecewise analytic arcs joining branch points of $F$, and some accessory points (similar to those of the Schwarz-Christoffel formula) associated with junctions  of these analytic arcs. For an example see Figure \ref{fig:two-cc-cut}.
Pad\'e represents actual poles of $F$ by poles, and branch points by lines (either straight or curved arcs). The pole density converges in capacity to the equilibrium measure along the arcs, and this density is infinite at the actual branch points, resulting in accumulation of poles there.
  \end{enumerate}
 
}\end{Note}

\subsubsection{Potential Theory and Physical Interpretation of Pad\'e Approximants}
  There is a remarkable and intuitively useful physical interpretation, which can be derived from \cite{Stahl,Saff}, of the domain $\mathcal{D}$ and of the placement of poles of Pad\'e.  We summarize the main aspects relevant for our analysis here:
  
   \begin{enumerate}

  \item Take any set $\mathcal{D}'$ of single-valuedness of $F$ and let $E'=\partial \mathcal{D}'$ be its boundary. Thinking of $E'$ as an electrical conductor we place a unit charge on $E'$, and normalize the electrostatic potential $V(x,y)=V(\omega), \omega=x+iy$ (always constant along a conductor) by $V(E')=0$. Then the electrostatic capacitance of $E'$ is cap$(E')= 1/V(\infty)$.

  \item The domain boundary $E=\partial \mathcal{D}$ of the domain of convergence of Pad\'e is obtained by deforming the shape (keeping the singularity locations fixed)  of the conductor $E'$ (defined in item 1 of this Note) until it has minimal capacity.
 
  \item  The equilibrium measure $\mu$  on $E$ is the equilibrium density of charges on $E$ in the setting above. As $j\to\infty$ the poles of the near diagonal Pad\'e approximants place themselves (except for a set of zero capacity) close to $E$, and Dirac masses placed at these poles converge in measure to $\mu$  \cite{Stahl}. 
  
  \item For $\omega\in \mathcal{D}$, we have $e^{-g_{\mathcal{D}}(\omega)}=|G_{\mathcal{D}}(\omega)|=e^{-V(\omega)}$. 
  
  \item
A brief summary of an associated numerical construction is outlined in the  Appendix \ref{sec:app}.

  \end{enumerate}

\subsection{New Approximate Methods for Detecting Hidden Singularities}
\label{sec:hidden}

It is not uncommon that discrete singularities of a function lie on the capacitor of Pad\'e approximants where they diverge, and  therefore cannot be seen in this way. All resurgent functions coming from differential equations have their singularities along half-lines starting from the origin, and symmetry reasons generally make those rays part of the capacitor. This is the case, for instance, for the {\em tronqu\'ee} Painlev\'e transcendents, P$_{I}$--P$_{V}$. In general the leading singularity is a branch-point, and, to ensure single-valuedness, Pad\'e ``creates'' a cut,  part of the capacitor.
Two ways to detect such  ``hidden'' singularities are described here.

\begin{enumerate}
	
	\item {\bf Probe Singularity Method:}
	
	The simplest  method is to place an artificial probe singularity near the arc. By the potential theory interpretation of Pad\'e we know that this additional singularity will distort the minimal capacitor, but it cannot move the genuine singularities. This simple procedure can be implemented as follows: assume that $J$ is an analytic arc of the  Pad\'e approximants of the  function $F\in \mathcal G$. Define  a new function $F_1(\omega)=\omega\mapsto F(\omega)+(\omega-\omega_0)^\alpha$, where $\alpha\notin\ZZ$ (a negative power  is typically more effective), such that the extra ``probe'' singularity at $\omega=\omega_0$ is placed in the proximity of the arc $J$. Clearly, the Pad\'e approximants of $F_1$ determine the values of $F$ as well, simply by subtracting out $(\omega-\omega_0)^\alpha$. The capacitor of $F_1$ is necessarily different from that of $F$, since the probe singularity $\omega_0$ must be part of the new capacitor. Generically, the arc $J$ moves when $\omega_0$ is chosen near any point of $J$ which is a point of analyticity of $F$. Evidently too, points in $J$ which are branched singularities of $F$ cannot move. 
	
	\item {\bf Conformal Mapping Method:}
	
	The second method consists of applying a form of CT. Any nontrivial conformal map of domains in $\CC$  changes the capacitor and  typically distorts all the arcs of the Pad\'e capacitor exposing previously hidden singularities, and possibly hiding ones that were visible before, and exposing domains that lie on the second Riemann sheet relative to the cut $\partial \mathcal{D}$.

	In the large $n$ limit, CT provides a {\bf rigorous} way to check the information inferred from a Pad\'e analysis. Indeed, conformally mapping a mistaken domain (or parts of a Riemann surface), results in singularities in $\DD$, seen in an $n$-th root test of CT (cf. \S\ref{Sempir}).

	The singularities of $F$ that lie on the boundary $\partial \mathcal{D}$ are mapped onto the unit circle  $\TT$, the boundary of $\DD$, and can therefore be resolved using a discrete Fourier transform of the properly normalized Maclaurin coefficients.
	
\end{enumerate}
\subsection{Example of Approximate Extrapolation: Painlev\'e equations PI-PV}
\label{S:ExampleP1}

The {\em tronqu\'ee} Painlev\'e transcendents are resurgent functions.
The Painlev\'e equations, $P_I$-$P_V$, have a common and simple Borel singularity structure, which can be arranged as integer-spaced singularities along the real line (excluding the origin). Even the Conformal-Pad\'e method, based on the simple two-cut conformal map  in (\ref{eq:twocutmap}), leads to a remarkably accurate extrapolation of the formal solution generated at infinity, throughout the complex plane: see \cite{Costin:2019xql} for a detailed analysis of the {\it tritronqu\'ee} solution of $P_I$. 

But with uniformizing maps significantly better extrapolation and analytic continuation can be achieved. 
Here we show that it is not necessary to use the {\it exact} uniformizing map from \S\ref{sec:P1-unif} in order to achieve highly accurate analytic continuation. One can instead use a crude approximation to the uniformization, based simply on the two (symmetric) leading Borel singularities, ignoring all the further integer-repeated Borel singularities. Recall Figure \ref{fig:P1singularities} for Painlev\'e I.
For example, even the simple step of replacing the two-cut conformal map  (\ref{eq:twocutmap}) with the two-puncture uniformizing map of $\Omega\left(\hat\CC\setminus \{-1, 1, \infty\}\right)$ in (\ref{eq:2cut}) leads to a dramatic improvement. This is illustrated in Figure \ref{fig:p1poles}.
\begin{figure}[htb]
	\centering
	\includegraphics[scale=1]{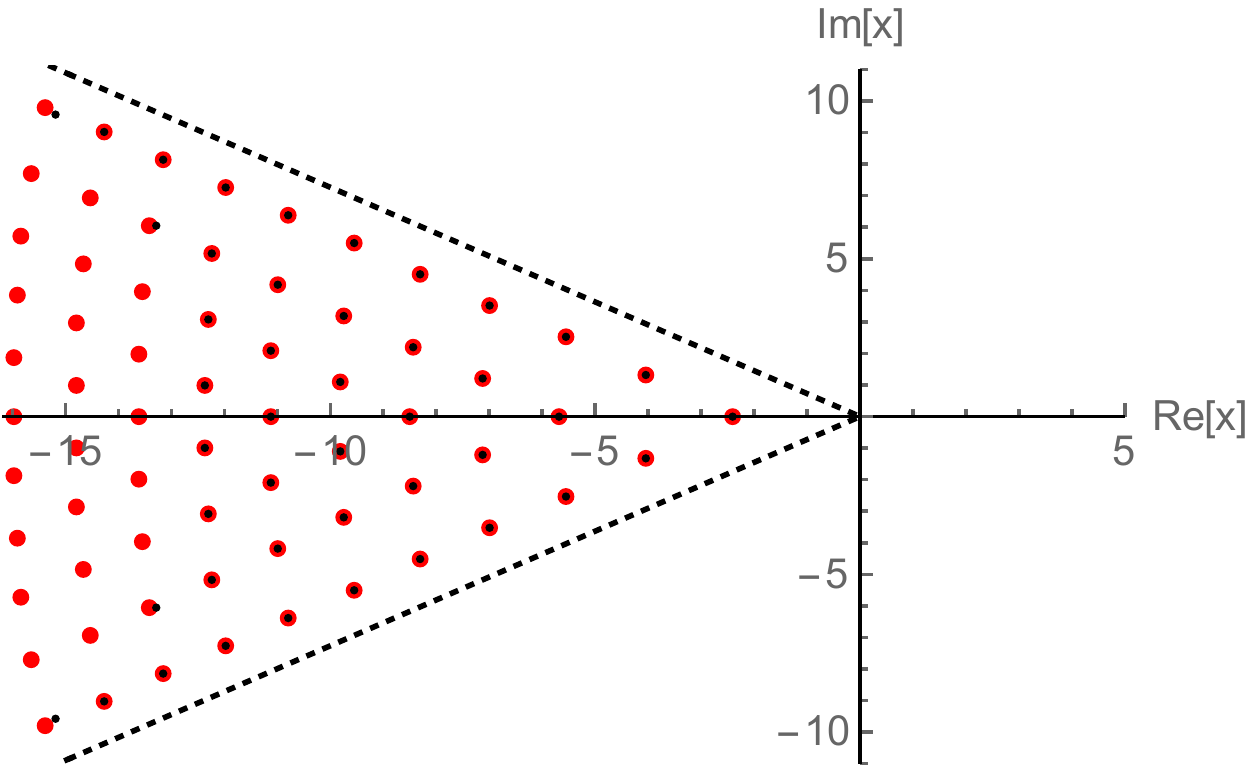}
	\caption{The poles of the {\it tritronqu\'ee} solution of the Painlev\'e 1 equation, which lie only in the wedge $\frac{4\pi}{5}\leq {\rm arg}(x)\leq \frac{6\pi}{5}$ of the complex plane, obtained by extrapolation of the asymptotic expansion about $x\to +\infty$, using 200 input coefficients. The smaller black dots show the results using the Conformal-Pad\'e method in the Borel plane, as described in  \cite{Costin:2019xql}, while the larger red dots are obtained by replacing the   conformal map (\ref{eq:twocutmap}) with the uniformizing map  (\ref{eq:2cut}).
		With exactly the same input data, the uniformizing map leads to a significantly better extrapolation. There is comparable accuracy throughout the domain of analyticity, $|{\rm arg}(x)|\leq \frac{4\pi}{5}$.}
	\label{fig:p1poles}
\end{figure}
Recall that the solution of the $P_I$ equation, $y''(x)=6\, y^2-x$, is meromorphic throughout the complex plane, and the special {\it tritronqu\'ee}  solution has poles only in the wedge, $\frac{4\pi}{5}\leq {\rm arg}(x)\leq \frac{6\pi}{5}$ \cite{dubrovin,Dubrovin}.
A nontrivial test of the precision of an extrapolation
is to reconstruct the $P_I$  {\it tritronqu\'ee} solution throughout its domain of analyticity, and also in its pole sector $\frac{4\pi}{5}\leq {\rm arg}(x)\leq \frac{6\pi}{5}$, using only input from its asymptotic expansion about the opposite direction,  $x\to +\infty$. 
Figure \ref{fig:p1poles} shows as black dots
$44$ $P_1$ {\it tritronqu\'ee} poles found using the Conformal-Pad\'e approach of  \cite{Costin:2019xql}, starting with $200$ terms of the asymptotic expansion generated at $x\to+\infty$, while the red dots show the first $66$ $P_1$  {\it tritronqu\'ee} poles found simply by adapting the analysis  of \cite{Costin:2019xql} to use the uniformizing map (\ref{eq:2cut}) instead of the conformal map in (\ref{eq:twocutmap}), and with exactly the same input data. The gain in precision in the pole sector, and also throughout the domain of analyticity, is quite dramatic. These numerical poles,  even the first few ones, fit very precisely the asymptotic Boutroux structure \cite{kitaev,Costin:2019xql}.

In addition, the  reconstruction using the uniformization explained above yields high-precision  fine structure of the pole region.
In the vicinity of a movable pole, say $x=x_j$,  any $P_I$ solution $y(x)$ has a Laurent expansion of the following form
\begin{eqnarray}
	y(x)= \frac{1}{(x-x_j)^2}+\frac{x_j}{10}(x-x_j)^2+\frac{1}{6}(x-x_j)^3+ h_j (x-x_j)^4
	+\frac{x_j^2}{300}(x-x_j)^6 +\cdots
	\label{eq:general}
\end{eqnarray}
where the constants $x_j$ and $h_j$, important in applications, are not determined by the equation. The coefficients of $(x-x_j)^{k},k>4$ are  expressed as polynomials in the two parameters $x_j$ and $h_j$. The {\it tritronqu\'ee} is completely determined by the constants $x_j$ and $h_j$  at any pole; the one closest pole to the origin, $j=1$ is particularly important.
From the  procedure explained above, we obtain the following high-precision values: 
\begin{eqnarray}
	x_1&=& -2.38416876956881663929914585244876719041040881473785051267725...
	\label{eq:x1}
	\\
	h_1&=& 0.0621357392261776408964901416400624601977407713738296636635333...
	\label{eq:h1}
\end{eqnarray}
These are significantly higher precision than existing values \cite{novokshenov2}.
Furthermore,
it is straightforward to  obtain even higher precision, if desired.  Similar methods apply to the other 
Painlev\'e {\it tronqu\'ee} solutions, providing new methods to obtain high-precision computations for the Painlev\'e project \cite{PainleveProject}, and also to compute high-precision spectral properties of certain Schr\"odinger operators \cite{masoero,novokshenov2}.

\section{Comparison to existing techniques used in the physics literature}\label{S:Applic}

\subsection{Conformal-Taylor (CT) and Conformal-Pad\'e (CP)} \label{CT}

We compare the accuracy of two methods that have been used in the physics literature. While they have been used rather infrequently and without convergence analysis, they can be quite useful to reach points outside of $\DD$. In the Conformal-Taylor (CT) method (as defined above in \S \ref{sec:t1discussion}) a domain $\mathcal{D}\subset \CC$ of analyticity of $F$ is chosen, and then one proceeds as in Theorem \ref{T1}, with $\mathcal{D}$ in guise of $\Omega$. The Conformal-Pad\'e method (CP)
\cite{ZinnJustin:2002ru,caliceti,Costin:2019xql,Costin:2020hwg,Costin:2021bay} consists of  a further step of applying Pad\'e approximants to  CT, which typically results in a significant increase in accuracy, at the price of having convergence in capacity only. Surprisingly, the CP  method appears to have been used even less frequently than CT. 
\begin{Note}{\rm 
  \begin{enumerate}
  
  \item If the domain $\mathcal{D}$ happens to be the maximal domain of analyticity, then of course, Theorem \ref{T1} shows that CT is optimal.
 
  \item  The error control of each of CT and CP  approximation is obtained from the map $\psi$ as in Theorem \ref{T1}.
   
  \item Also as a consequence of Theorem \ref{T1},  CT can be improved by choosing a domain $\mathcal{D}$ with cap($\mathcal{D}$) as small as possible within the class of domains having explicit conformal maps $\psi$.

  \item To expand the class of explicit maps, we note that one only needs a map $\phi:\DD\to \mathcal{D}$ which is {\bf surjective} (at the price of a slower rate of reconstruction of $F$). 
  
   \end{enumerate}
  
}\end{Note}

\subsection{Improvement of Uniformization over Pad\'e and Conformal-Pad\'e (CP)}

The generic improvement of analytic continuations based on uniformization maps, compared with other common methods such as Pad\'e or  Conformal-Pad\'e (CP) [as illustrated in the previous section], can often be traced to some elementary properties of these maps,  especially near the singularities. Here we illustrate this with some examples.

\subsubsection{One Cut Complex Plane}
\label{sec:1cut}

For $\Omega=\CC \setminus [1, \infty)$, the conformal map is in (\ref{eq:onecutmap}), and the uniformizing map of $\hat\CC\setminus \left\{1, \infty\right\}$ is $\omega=1-e^{-z}$, with inverse $z=-\log(1-\omega)$. The uniformizing map pushes the singularity at $\omega=1$ to $z=\infty$. A Pad\'e approximant of the truncated  composed map $F\circ \phi$, with either the conformal map or the uniformizing map, each mapped back to $\Omega$, produces dramatically improved extrapolations throughout $\Omega$. 
 For example, for the one-branch-cut function $F(\omega)=(1-\omega)^{-1/5}$, beginning with just 10 terms of a Maclaurin expansion at $\omega=0$, Figure \ref{fig:one-cut} shows the ratio of the extrapolated to the exact function, as the singularity is approached: $\omega\to 1^-$. The uniformization map is vastly superior.  This is due to the typical exponential distortion of distance near the singularity. Ordinary Pad\'e, without composition with either map, is not at all competitive. 
\begin{figure}[h!]
  \centering\includegraphics[scale=.75]{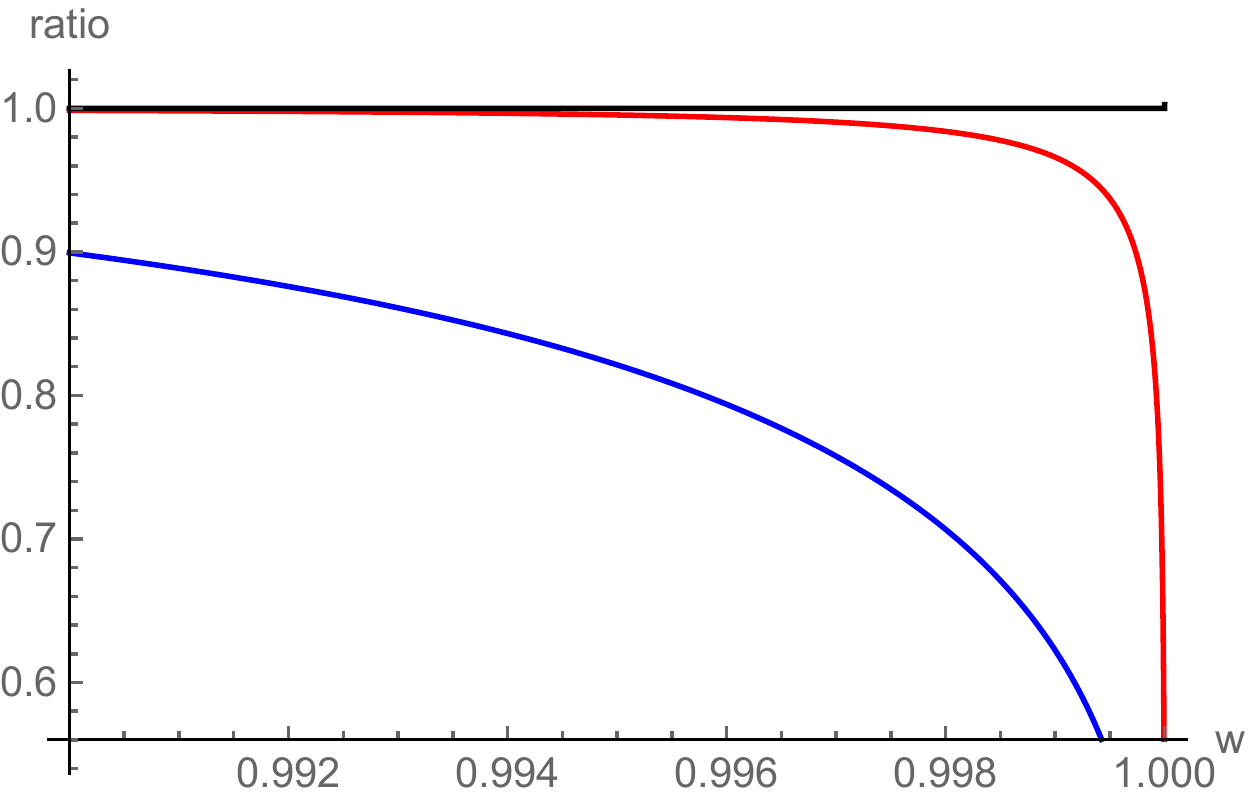}
   \caption{Ratio of the approximate to the exact one-branch-cut function $F(\omega)=(1-\omega)^{-1/5}$, for a Pad\'e approximation [blue], a Pad\'e-Conformal approximation [red], and a Pad\'e-Uniformized approximation [black]. Note the dramatically superior behavior of the uniformized approximation as the singularity is approached: the ratio is essentially 1, except for a tiny blip very close to $\omega=1$. These approximations are each generated starting with just 10 input coefficients of the series expansion of $F(\omega)$ at $\omega=0$.}
  \label{fig:one-cut}
\end{figure}

\subsubsection{Two Cut Complex Plane}
\label{sec:2cut}

 Another common case in applications is the two-cut complex plane, $\Omega=\CC \setminus (-\infty, -1]\cup  [1, \infty)$. The conformal and uniformizing maps are given in (\ref{eq:twocutmap}) and (\ref{eq:2cut}), respectively,  in  \S \ref{sec:app}. For example, applying this to the two-branch-cut function $F(\omega)=(1-\omega^2)^{-1/5}$, with branch points at $\omega=\pm 1$, produces similar improvements as in \S\ref{sec:1cut}. In this case the conformal map sends the cut $\omega$ plane to the interior of the unit disk in the $z$ plane, while the uniformizing map sends the cut $\omega$ plane to the interior of the symmetric geodesic quadrilateral with boundaries given by orthogonal circles intersecting the unit disk at $z=\pm 1, \pm i$. See the left plot in Figure \ref{Fig:nonconfplot64}.

Even when these are not the exact conformal or uniformizing maps, such as in nonlinear problems where the singularities at $\omega=\pm 1$ are only the leading ones, generally repeated at all non-zero integers, the use of these two-cut maps leads to dramatic improvements, especially in the vicinity of these leading singularities
 \cite{Costin:2019xql,Costin:2020hwg,Costin:2021bay} as in \S\ref{S:ExampleP1}. This example is particularly relevant in applications, as there are many examples where the (Borel) singularities appear in integer multiples along oppositely directed straight lines: for example, the Borel plane of:
the {\em tronqu\'ee} Painlev\'e I-V solutions, the Euler-Heisenberg effective action, rigorously known, \cite{Duke,dunne-eh}, and renormalon singularities in quantum field theory (highly numerically corroborated)
 \cite{beneke}. 
  \begin{figure}[h!]
  \centering\includegraphics[scale=.75]{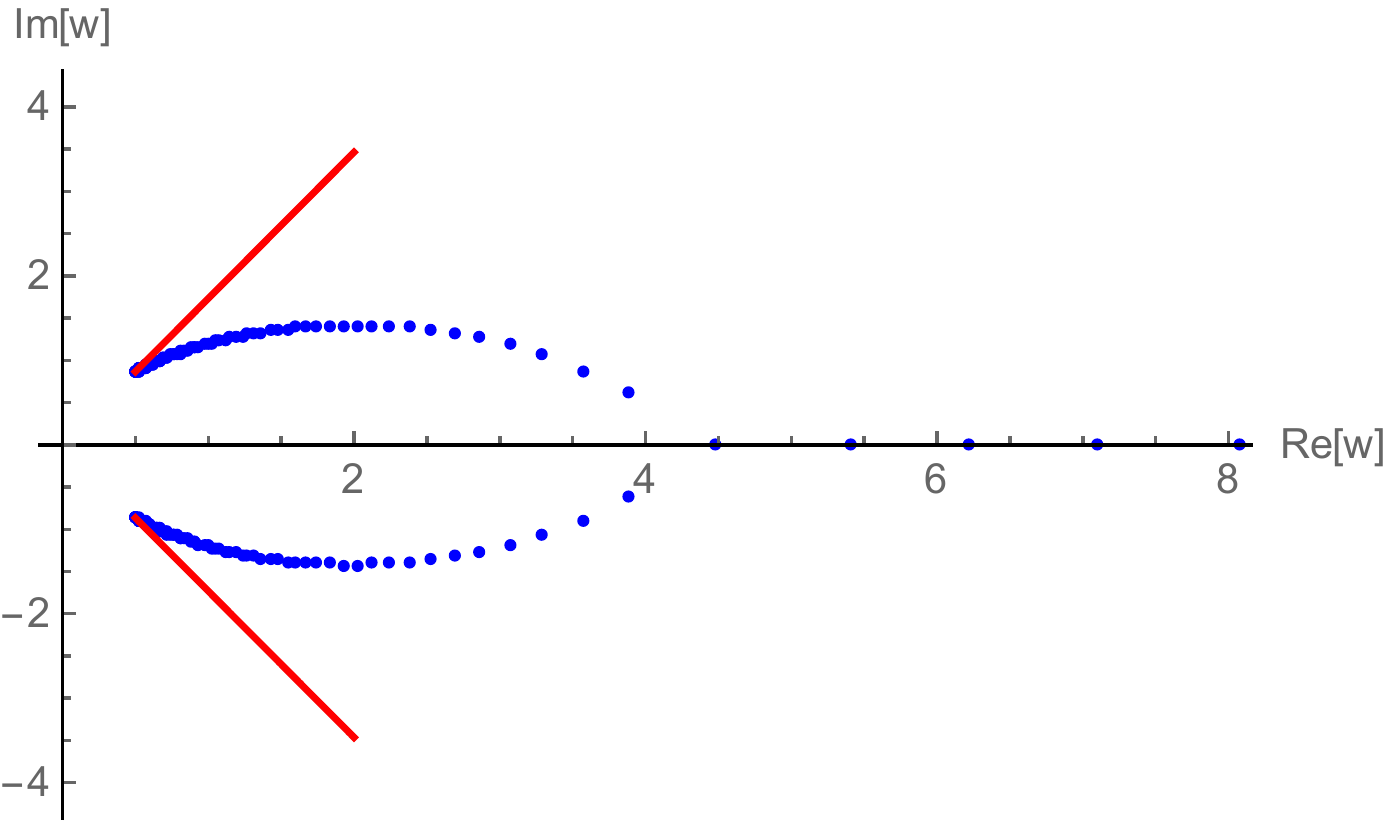}
   \caption{
   Arcs of Pad\'e poles (blue points)  for a pair of complex conjugate singularities, here at $\omega=e^{\pm i \pi/3}$, for the function $F(\omega)=(1-2\omega \cos(\pi/3)+\omega^2)^{-1/5}$.  Pad\'e generates  ``unphysical'' arcs of poles (blue) along its minimal capacitor, including an eventually dense set on part of $\RR^+$, while the ``natural'' radial cuts (red lines) are associated with the  conformal map (\ref{eq:cc1}).}  
  \label{fig:two-cc-cut}
\end{figure}
Another important configuration for applications consists of two complex conjugate singularities at $\omega=e^{\pm i \theta}$.
This type of configuration occurs in physical applications involving a parameter that breaks the collinear symmetry \cite{stephanov,bertrand,rossi,serone,heller}. In this case, without a conformal or uniformizing map, Pad\'e produces curved arcs of poles (see Figure \ref{fig:two-cc-cut}) as well as artificial poles along the positive real axis.  These poles are not related to the analytic properties of the function which is being approximated. Problems due to these artificial poles can be mitigated by using a conformal or uniformizing map. The  conformal map  for this configuration is
  \begin{figure}[h!]
  \centering\includegraphics[scale=.75]{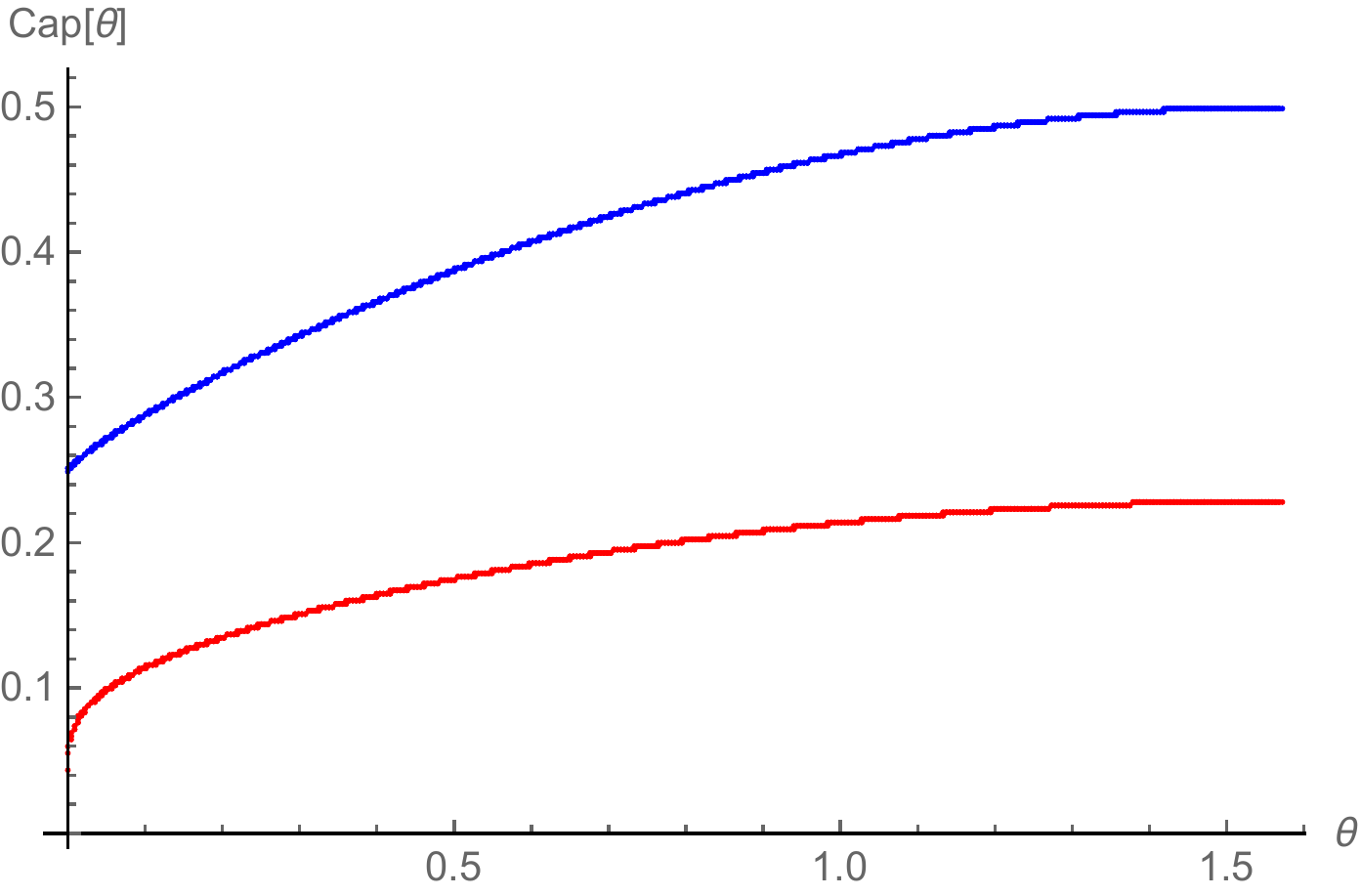}
   \caption{Plot of the acceleration modulus $a_c$ associated with (\ref{eq:conf-cap})  (blue curve) and with  (\ref{eq:unif-cap}) (red curve), $a_u$,  as a function of $\theta\in [0, \pi/2]$, for the configuration with two complex conjugate singularities at $e^{\pm i \theta}$. 
}
  \label{fig:cap}
\end{figure}
\begin{eqnarray}
\omega= c(\theta)\frac{z}{(1+z)^2}\, \left(\frac{1+z}{1-z}\right)^{2\theta/\pi}
\qquad, \qquad c(\theta)=4\left(\frac{\theta}{\pi}\right)^{\theta/\pi}\left(1-\frac{\theta}{\pi}\right)^{1-\theta/\pi} 
\label{eq:cc1}
\end{eqnarray}
and the uniformizing map is given by \eqref{eq:UnifSymmetric}.
The acceleration moduli for the two cases are
\begin{eqnarray}
\text{conformal :}\qquad a_c(\theta)&=&\frac{1}{4} \left(\frac{\theta}{\pi}\right)^{-\frac{\theta }{\pi }} \left(1-\frac{\theta }{\pi }\right)^{\frac{\theta }{\pi }-1}
\label{eq:conf-cap}\\
\text{uniformizing :}\qquad  a_u(\theta)&=&\frac{i\,\frac{\pi}{2}\, \sin\theta}{\left(\mathbb K\left(\frac{1}{2}+\frac{i}{2}\cot \theta\right)\right)^2 + \left(\mathbb K\left(\frac{1}{2}-\frac{i}{2}\cot \theta\right)\right)^2}
\label{eq:unif-cap}
\end{eqnarray}
As shown in Figure \ref{fig:cap},  the uniformizing map has a better acceleration modulus for all $\theta$, and also has an explicit inversion. The conformal map is a simpler elementary function, but has no explicit inversion except for a few special cases of  rational $\theta/\pi$. A simple Pad\'e approximation, which produces the minimal capacitor (for which there exist implicit transcendental expressions for the acceleration modulus -- the minimal capacity
 \cite{kuzmina,grassmann}),  is inferior to both the conformal and uniformizing maps.

\section{Appendix: some relevant conformal maps}
\label{sec:app}

In this Appendix we record some conformal maps relevant for frequently encountered cases of resurgent functions, and which can also serve as guides in  more complicated arrangements of singularities. General conformal maps can in principle be derived from Schwarz-Christoffel,  but this procedure is rather tedious, and in cases of symmetry the resulting maps can be quite elementary. For a comprehensive list of many known conformal maps, see \cite{Kober}.

A simple but important case is the one-cut domain $\Omega=\mathbb{C}\setminus [1,\infty)$, for which
\begin{equation}
  \label{eq:onecutmap}
z=  \psi(\omega)=\frac{1-\sqrt{1-\omega}}{1+\sqrt{1-\omega}}
  \qquad \text{with inverse }\qquad  
 \omega= \phi(z)=\frac{4z}{(1+z)^2}
\end{equation}
The optimal rate of convergence obtained from  the Maclaurin series of a generic function $F$ whose maximal analyticity domain is this $\Omega$, and is continuous up to $\partial \DD$ is, see \eqref{eq:CauchyEst},
$$|F(\omega_0)- \hat{R}_n(\omega_0) |\sim \frac{|\omega_0|^n}{2\left |\sqrt{1-\omega_0}\right | \left |1+\sqrt{1-\omega_0}\right |^{2n-1}} \|F\|_{\infty};\qquad  \omega_0\in\Omega$$
The constant $C=1/4$ in (\ref{oeq}), the capacity of $1/\partial \Omega$, is simply $\psi'(0)$.

For the domain with two opposite cuts,  $\Omega=\mathbb{C}\setminus (-\infty,-1]\cup[1,\infty)$, the maps are
\begin{equation}
  \label{eq:twocutmap}
z=  \psi(\omega)=\sqrt{\frac{1-\sqrt{1-\omega^2}}{1+\sqrt{1-\omega^2}}}
\qquad \text{with inverse }\qquad  
 \omega=  \phi(z)=\frac{2z}{1+z^2}
\end{equation}
with $\psi(\omega)>0$ for $\omega\in (0,1)$. The capacity of $1/\partial \Omega$ is now $C=1/2=\psi'(0)$. This construction generalizes straightforwardly to $m$ symmetric cuts emanating from the vertices of a regular polygon. See  Appendix \ref{sec:app} and  \cite{Hempel}. This example also generalizes to $\Omega=\mathbb{C}\setminus (-\infty,-a]\cup[b,\infty)$: see  (\ref{eq:cmap-ab}) in Appendix \ref{sec:app}.
\begin{Note}{\rm 
There is a more general principle behind \eqref{eq:twocutmap} worth mentioning:
      \begin{Lemma}
      \label{L:n-fold}
         Let $\Phi$
by a conformal map of $\DD$ to some domain $\mathcal D\subset \CC$, and let $c=\Phi'(0)>0$.  Then, 
$\Phi_n(z):=\Phi(z^n)^{1/n}$  maps $\DD$ conformally to $n$ symmetric
copies of $\mathcal D^{1/n}$, i.e., to
$$\bigcup_{0\le j\le n-1}e^{2\pi i j/n}\mathcal D^{1/n}$$
      \end{Lemma}
     \begin{proof}
In a neighborhood of zero, $\Phi_n$ is uniquely defined by
$\Phi_n(z)=|c^{1/n}|z H(z^n)$, where $H$ is analytic at zero and
$H(0)=1$. Since $\Phi\ne 0$ on $\DD\setminus\{0\}$, by the monodromy
theorem, $\Phi_n$ extends analytically to $\DD$. Since $\Phi$ is
injective on $\DD$,   $\Phi_n(z)=\Phi_n(v)$ implies $z^n=v^n$. Now,
$\Phi_n(z)=\Phi_n(v)$, written as $|c^{1/n}|z H(z^n)=|c^{1/n}|v H(v^n)$
implies $z=v$, and thus $\Phi_n$ is injective. Since $\Phi$ is onto
$\mathcal{D}$, the rest follows from injectivity.
\end{proof}
With proper adaptations, this construction extends to uniformization maps of Riemann surfaces.}
\end{Note}

{\bf \refstepcounter{minisection} \label{mini3} \arabic{minisection}.} For functions analytic on $\Omega\left(\hat\CC\setminus \{-1,1, \infty\}\right)$ the maps are (compare with \cite{Bateman}, p. 99)
\begin{eqnarray}
z= \psi(\omega)=\frac{\mathbb K\left(\frac{1+\omega}{2}\right)-\mathbb K\left(\frac{1-\omega}{2}\right)}{\mathbb K\left(\frac{1-\omega}{2}\right)+\mathbb K\left(\frac{1+\omega}{2}\right)} 
\qquad \text{with inverse }\qquad  
\omega =  \phi(z)=-1+2 \lambda \left(i\, \frac{1-z}{1+z}\right)
\label{eq:2cut}
\end{eqnarray}
Here $\lambda=\theta_2^4/\theta_3^4$ is the elliptic modular function, $\theta_2,\theta_3$ are Jacobi theta functions, and $\mathbb K(m)=(\pi/2) \,_2F_1(\tfrac12,\tfrac12;1;m)$ is the complete elliptic integral of the first kind of modulus $m=k^2$  \cite{Bateman}. The capacity is $C=\psi'(0)=\pi^{-2}\Gamma \left(\frac{3}{4}\right)^4\approx 0.2285$, more than a factor of two better than the capacity of the conformal map in Equation (\ref{eq:twocutmap}) of Example 2 above. Furthermore, $\hat{R}_n(\omega_0)$ in \eqref{eq:CauchyEst} converges on the whole universal covering of $\hat\CC\setminus \{-1,1, \infty\}$; that is, on all the Riemann sheets of the underlying function. See Figure \ref{Fig:nonconfplot64}.

Note that the improvement in accuracy is particularly dramatic near singular points. Indeed, the leading order asymptotic behavior of $\psi$ near $\omega=1$ is $\psi(\omega)\sim 1+2\pi/\ln(1-\omega)$.

\begin{figure}[h!]
  \centering\includegraphics[scale=0.28]{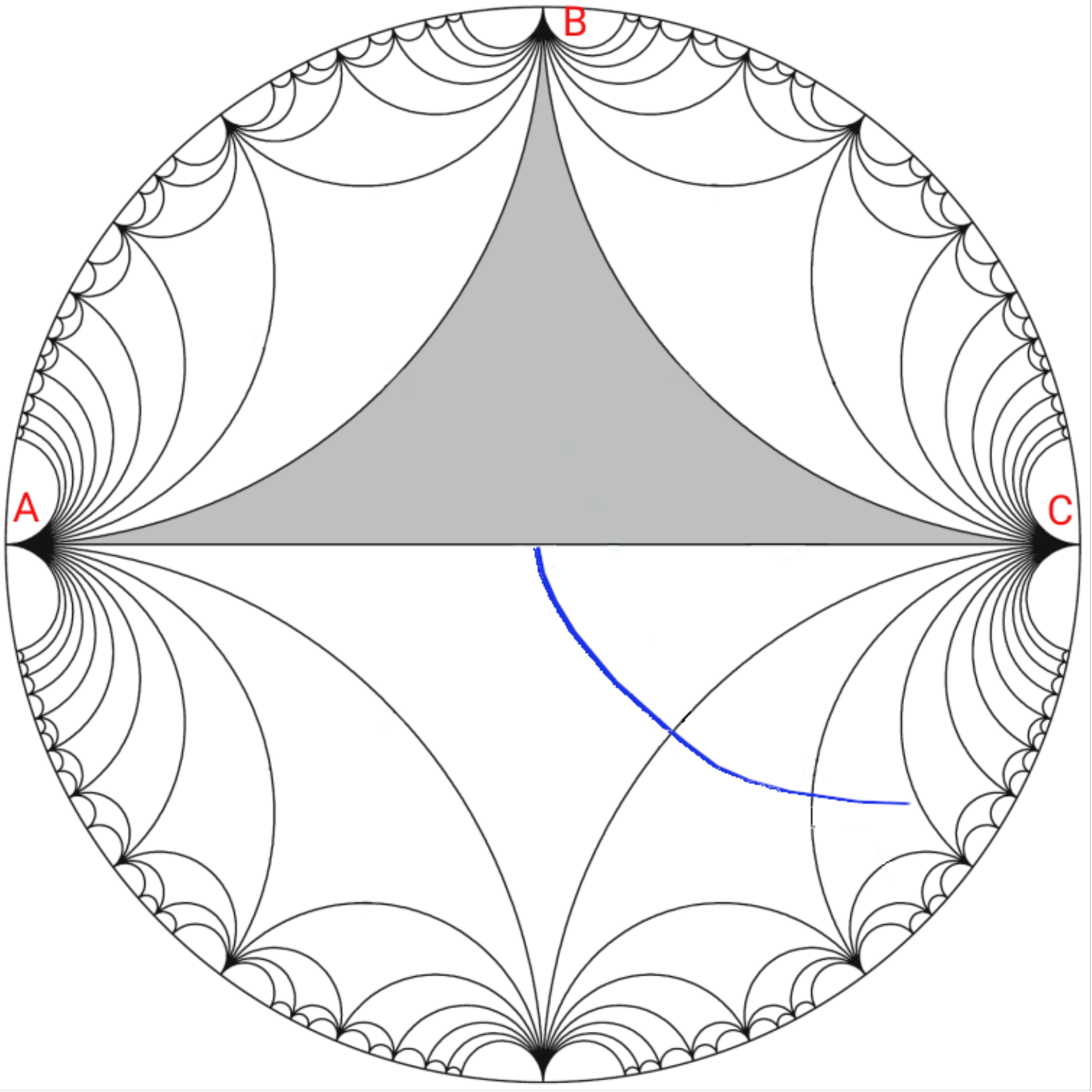} \includegraphics[scale=0.55]{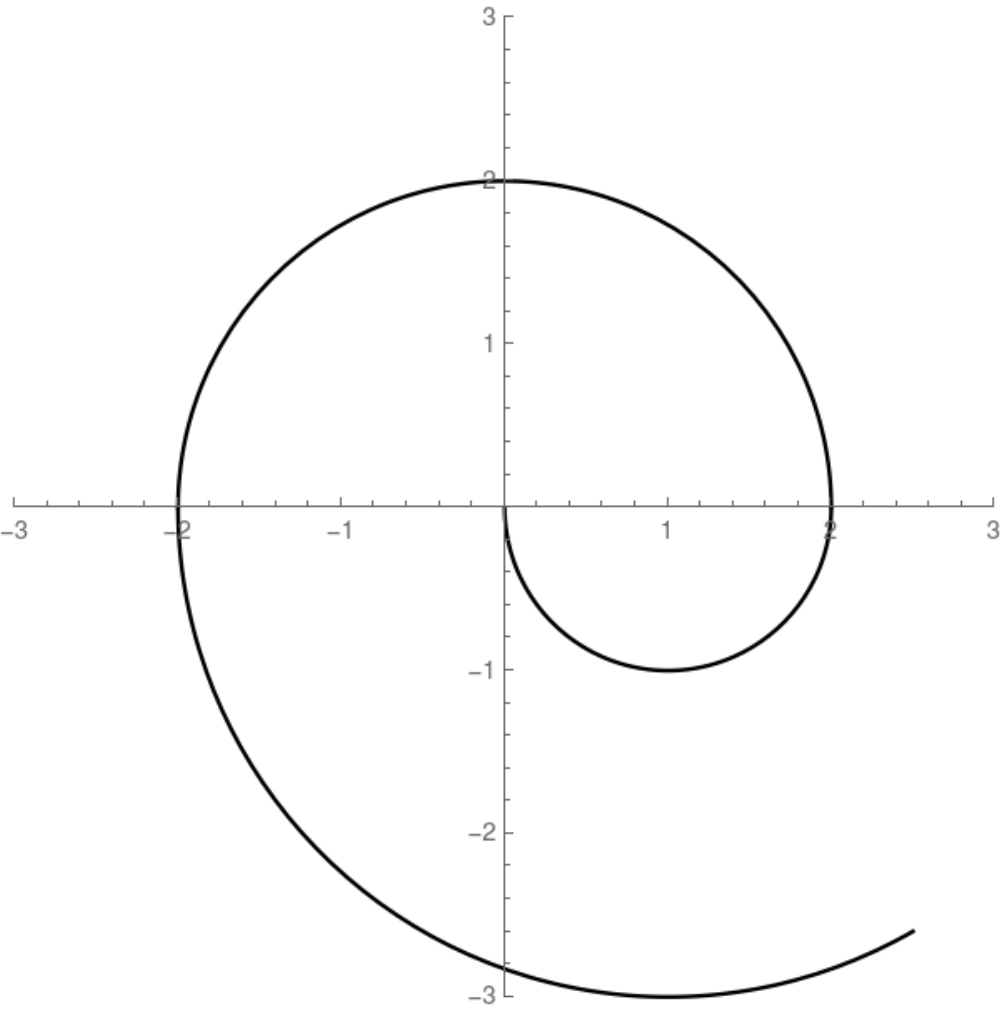}
   \caption{
    Uniformization of $\Omega\left(\hat\CC\setminus \{-1,1, \infty\}\right)$ by the map $\psi$ in (\ref{eq:2cut}).  The blue curve in the left-hand figure corresponds to the spiral path on $\Omega\left(\hat\CC\setminus \{-1,1, \infty\}\right)$, which crosses to two higher sheets.}
\label{Fig:nonconfplot64}
\end{figure}

The function $\psi$ in \eqref{eq:2cut} maps the points $A,B,C$  to $\{-1,\infty,1\}$. The gray geodesic triangle in the Poincar\'e disk is conformally mapped by $\phi$ onto the upper half plane, and successive Schwarz reflections across their circular sides continue $\phi$ to the whole disk, with image onto $\Omega$ (see \cite{Ablowitz}, p. 379).
The union of all reflected triangles is the unit disk $\DD$. The image through $\phi$ of a  curve in $\DD$ crossing a reflection of $(A,B), (B,C), (C,A)$ crosses the real line between $(-\infty,-1), (1,\infty), (-1,1)$ resp. The collection, modulo homotopies, of images  through $\phi$ of all the curves in $\DD$ (each of them traceable using this geometric description) represents $\Omega\left(\hat \CC\setminus \{-1, 1, \infty\}\right)$.

In the left figure of Figure \ref{Fig:nonconfplot64}, the blue path inside the disk is mapped to the spiral path in the middle figure on $\Omega\left(\hat \CC\setminus \{-1, 1, \infty\}\right)$. The improvement of the rate of convergence is determined by the conformal distance to the boundary of the unit disk, and is $\sim 0.228^n$ near zero, and the rate is about $0.83^n$ at the other end of the spiral.

{\bf \refstepcounter{minisection} \label{mini4.0} \arabic{minisection}. } The uniformization map for $\hat\CC\setminus \{e^{-i\theta},e^{i\theta},\infty\}$, another important case in applications,  is given by  (we consider $0\leq \theta\leq \frac{\pi}{2}$, and extend by symmetry)
\begin{eqnarray}\label{eq:UnifSymmetric}
z=\psi(\omega)=\frac{Z(\omega; \theta) -Z(0; \theta)}{1-\left(Z(0; \theta)\right)^*\, Z(\omega; \theta)}
\label{eq:cc2}
\end{eqnarray}
where $Z(\omega; \theta)$ is defined as
\begin{eqnarray}\label{eq:UnifSymmetricZ}
Z(\omega; \theta)\equiv \frac{\mathbb K\left(\frac{1}{2}+\frac{i}{2}\left(\frac{\omega}{\sin\theta}-\cot \theta\right)\right)
-\mathbb K\left(\frac{1}{2}-\frac{i}{2}\left(\frac{\omega}{\sin\theta}-\cot \theta\right)\right)}{\mathbb K\left(\frac{1}{2}+\frac{i}{2}\left(\frac{\omega}{\sin\theta}-\cot \theta\right)\right)
+\mathbb K\left(\frac{1}{2}-\frac{i}{2}\left(\frac{\omega}{\sin\theta}-\cot \theta\right)\right)}
\label{eq:cc2Z}
\end{eqnarray}
The inverse map is given in terms of the modular $\lambda$ function by
\begin{eqnarray}
\omega= \phi(z)= e^{i\, \theta} -2i\,\sin(\theta)\,  \lambda\left(i\left(
\frac{\mathbb K\left(\frac{1}{2}+\frac{i}{2} \cot\theta\right)-\mathbb K\left(\frac{1}{2}-\frac{i}{2} \cot\theta\right) z}
{\mathbb K\left(\frac{1}{2}-\frac{i}{2} \cot\theta\right)+\mathbb K\left(\frac{1}{2}+\frac{i}{2} \cot\theta\right) z}\right)\right)
\end{eqnarray}
These maps are obtained from \eqref{eq:2cut} by a suitable M\"obius transformation and disk automorphism.

{\bf \refstepcounter{minisection} \label{mini4} \arabic{minisection}. } It is straightforward to generalize the two previous examples to $\Omega\left(\hat\CC\setminus \{\omega_1, \omega_2, \infty\}\right)$, where $\omega_1, \omega_2 \in \CC$. The uniformization maps are again expressed in terms of the elliptic function $\mathbb K$ and the elliptic modular function $\lambda$. The important case in the previous example, with two complex conjugate points, $\omega_1=e^{i\theta}=1/\omega_2$, which occurs in many applications, is discussed further in \S \ref{sec:2cut}. 

{\bf  \refstepcounter{minisection} \label{mini5} \arabic{minisection}. }  For uniformization of other Riemann surfaces based on that of $\hat\CC\setminus \{-1,1, \infty\}$, possessing a nontrivial fundamental group, see for example \cite{Bateman}, \S 2.7.2. See also \cite{Hempel} for a collection of explicit uniformization maps of $\Omega\left(\hat\CC\setminus S\right)$, where $S$ is a finite set of points in $\hat\CC$, including for example $S=\{-1,0,1, \infty\}$, $S=\{-3,-1,0,1,3, \infty \}$ and $S=\{0,1, e^{\pi i/3}, \infty\}$, and  when $S$ consists of the $n$-th roots of unity. Uniformizing maps for the four-punctured torus are analyzed in \cite{keen}.
Algebraic functions have compact Riemann surfaces, and some explicit uniformizing maps can be found in Schwarz's table in \cite{Bateman}. 

{\bf  \refstepcounter{minisection} \label{mini6} \arabic{minisection}. }  In certain special cases the uniformizing map produces a meromorphic or rational function, in which case a subsequent Pad\'e approximation becomes exact. For example, the  function $F(\omega)= \sqrt{1-\omega}$ has a compact Riemann surface $\Omega$ (as all algebraic functions do). 
 The uniformizing map  $\omega=\phi(z)=4z/(1+z)^2$ makes $F\circ\phi$ meromorphic, $(F\circ\phi)(z)=(1-z)/(1+z)$,  hence analytic on the Riemann sphere, and Pad\'e $[n,n]$ is exact for $n>0$. A more sophisticated example is the Riemann surface $\Omega$ of functions with three square root branch points, which is uniformized by $(z^2-1)^2/(z^2+1)^2$, \cite{Bateman}, and the functions become rational; the uniformization theorem brings $\Omega$ to $\hat{\CC}$. This is another case where Pad\'e becomes exact.

The maps $\phi_1$ and $\phi_2$ in \eqref{eq:r1r2} take $\DD$ to $\CC\setminus [1,\infty)$ and 
$\CC\setminus (-\infty,-1]\cup [1,\infty)$, respectively. The latter case generalizes to two more general cuts on the real line, $\CC\setminus (-\infty,-a]\cup [b,\infty)$ with $a, b\in\RR^+$, for which
\begin{equation}
\omega=  \phi_{ab}(z)= \frac{4 a b\, z}{a (1+z)^2+b (1-z)^2} \quad \leftrightarrow \quad z=\frac{1-\sqrt{\frac{a (b-\omega)}{b (a+\omega)}}}{1+\sqrt{\frac{a (b-\omega)}{b (a+\omega)}}}
  \label{eq:cmap-ab}
\end{equation}
For a symmetric set of $n$ singularities on the unit circle, $\CC\setminus  \bigcup_{j=0}^{n-1} e^{2 \pi i j/n}[1,\infty)$, the conformal map producing symmetric radial cuts is
\begin{equation}
  \omega=\phi_n(z)=  \frac{2^{2/n} z}{(1+z^n)^{2/n}}
\label{eq:cmap-n}
\end{equation} 
used in \S\ref{sec:composition}. For two complex conjugate radial cuts, $\CC\setminus e^{ i\theta}[1,\infty)\cup e^{ -i\theta}[1,\infty)$, the conformal map is as in (\ref{eq:cc1}), and this extends to a symmetric set of $n$ such paired cuts as:
\begin{eqnarray}
z=\phi_{n, \theta}(z)= c_n(\theta)\frac{z}{(1+z^n)^{2/n} }\, \left(\frac{1+z^n}{1-z^n}\right)^{2\theta/\pi}
\quad, \quad c_n(\theta)=2^{2/n}\left(\frac{n\, \theta}{\pi}\right)^{\theta/\pi}\left(1-\frac{n\, \theta}{\pi}\right)^{1/n-\theta/\pi} 
\end{eqnarray}

For a general finite set of branch points, Pad\'e produces a conformal map in the infinite order limit, and this map corresponds to the minimal capacitor. Recall the discussion in \S\ref{S:empirical}. The analytic description of this minimal capacitor conformal map is as follows
 \cite{grassmann}.
Let $S=\{\omega_1,...,\omega_n\}$ be branch points, and $\mathcal C$ the minimal capacitor, with the point of analyticity placed at infinity. The conformal map $\phi$ that takes $\CC\setminus \DD$ to $\CC\setminus \mathcal{C}$, with $+\infty\mapsto+\infty$, has the Taylor expansion at infinity
\begin{equation}
  \label{eq:confomapinfi}
  \phi(\zeta)=C_B\zeta+\sum_{k=0}^\infty b_k \zeta^{-k}
\end{equation}
where $C_B$ is the capacity. There is a set $\{a_1,...,a_{n-2}\}\subset \mathcal C$ of auxiliary parameters
such that $\phi$ satisfies the equation
\begin{equation}
  \label{eq:confomapeq}
 \log \zeta=\int^{\phi(\zeta)}\sqrt{\frac{\prod_{j=1}^{n-2}(s-a_j)}{\prod_{j=1}^{n}(s-\omega_j)}}ds
\end{equation}
These auxiliary parameters are  the intersection points of the set of analytic arcs of $\mathcal{C}$, the limiting location set of the poles of the diagonal Pad\'e approximation $P_n[F]$  {\em for any} function $F$ having $S$ as the set of branch points, and being analytic in the complement of $\mathcal{C}$. (For example, in the infinite $n$ limit, in Figure \ref{fig:two-cc-cut} the point on the positive real axis near $\omega\approx 4.5$ would tend to the single intersection point for this configuration.)
The analytic arcs $\gamma_j$ ($\gamma'_j$, resp)  joining $a_1$ with $\omega_j, j=1,...,n-1$ ($a_1$ to $a_j, j=2,...,n-2$, resp.) are given by
\begin{equation}
  \label{eq:arcs}
  \Re\int_{\gamma_k}^{\omega}\sqrt{\frac{\prod_{j=1}^{n-2}(s-a_j)}{\prod_{j=1}^{n}(s-\omega_j)}}ds=0\ \ \text{and} \ \ \   \Re\int_{\gamma'_m}^{\omega}\sqrt{\frac{\prod_{j=1}^{n-2}(s-a_j)}{\prod_{j=1}^{n}(s-\omega_j)}}ds=0
\end{equation}
where $k=1,...,n-1$ and $m=2,...,n-2$. In cases of symmetrically distributed branch points, these integrals can be expressed in terms of elementary or elliptic functions 
 \cite{kuzmina}, and in more general cases the minimal capacitor produced by Pad\'e can be found numerically
 \cite{grassmann}. This construction benefits from physical intuition arising from the interpretation of the minimal capacitor  in terms of potential theory (see \S\ref{S:empirical}).

 \vspace{.3cm}
\noindent {\bf Acknowledgments} \\
 We thank R. Costin for numerous helpful discussions and comments. 
We also thank A. Voros for interesting discussions, and Jean \'Ecalle for detailed correspondence, and for valuable comments on an earlier draft.
This work is supported in part by the U.S. Department of Energy, Office of High Energy Physics, Award  DE-SC0010339 (GD).

\end{document}